\input ./style/arxiv-general.cfg
\documentclass[seceqn,MSNbibl,citesort,number,dvips]{arxbj}
\makeatletter
   \@ifpackageloaded{graphicx}{}{\usepackage{graphicx}}
\makeatother
\usepackage{mathrsfs}

\aid{0}
\volume{21}
\issue{4}
\pubyear{2015}
\firstpage{2157}
\lastpage{2189}
\doi{10.3150/14-BEJ639} 
\docsubty{FLA}

\makeatletter

\newcommand{\rrvert}{\vert}
\newcommand{\rrVert}{\Vert}
\newcommand{\llvert}{\vert}
\newcommand{\llVert}{\Vert}

\newtheorem{theorem}{Theorem}[section]
\newtheorem{lemma}[theorem]{Lemma}
\newtheorem{corollary}[theorem]{Corollary}

\newproclaim{definition}{Definition}[section]
\newremark{remark}[definition]{Remark}

\newcommand{\eqref}[1]{(\ref{#1})}

\renewcommand{\phi}{\varphi}
\newcommand{\eps}{\varepsilon}
\renewcommand{\epsilon}{\varepsilon}

\newcommand{\D}{\Delta}
\newcommand{\eq}{\eqref}

\newcommand{\bigo}{\mathrm{O}}

\newcommand{\toinf}{\to\infty}

\def\t#1{\tilde{#1}}

\def\cK{\mathcal{K}}
\def\cF{\mathcal{F}}
\def\cV{\mathcal{V}}

\newcommand{\mathbbm}{\mathbb}
\newcommand{\IE}{\mathbbm{E}}
\newcommand{\IP}{\mathbbm{P}}
\newcommand{\IR}{\mathbbm{R}}
\newcommand{\II}{\mathbbm{I}}
\def\Id{\mathbbm{I}}

\def\CI{\operatorname{CS}}
\def\endo{\operatorname{inj}}
\def\dist{\operatorname{dist}}
\newcommand{\Des}{\operatorname{Des}}
\newcommand{\Inv}{\operatorname{Inv}}
\newcommand{\Var}{\operatorname{Var}}
\newcommand{\Cov}{\operatorname{Cov}}
\newcommand{\Tr}{\operatorname{Tr}}
\newcommand{\Cor}{\operatorname{Cor}}
\newcommand{\Uniform}{\operatorname{Uniform}}

\newcommand{\law}{\mathscr{L}}

\def\E{\IE}
\def\P{{\IP}}
\def\I{\mathrm{I}}

\newcommand{\mrr}{\IR^d}

\renewcommand{\^}{\hat}

\newcommand{\fracd}[2]{({#1}/{#2})}

\newcommand{\fracb}[2]{(#1)/{#2}}
\makeatother

\begin{document}
\begin{frontmatter}

\title{Rates of convergence for multivariate
normal approximation with applications to dense graphs and
doubly indexed permutation statistics}
\runtitle{Rates of convergence for multivariate
normal approximation}

\begin{aug}
\author[A]{\inits{X.}\fnms{Xiao}~\snm{Fang}\thanksref{A}\ead[label=e1]{fangxiao@stanford.edu}} \and
\author[B]{\inits{A.}\fnms{Adrian}~\snm{R\"ollin}\corref{}\thanksref{B}\ead[label=e2]{adrian.roellin@nus.edu.sg}}
\address[A]{Department of Statistics, Stanford University, Stanford, CA
94305-4065, USA.\\ \printead{e1}}
\address[B]{Department of Statistics and Applied Probablity, National
University of Singapore, 6 Science Drive 2, Singapore 117546. \printead{e2}}
\end{aug}

\received{\smonth{6} \syear{2012}}
\revised{\smonth{4} \syear{2014}}

%
\begin{abstract}
We provide a new general theorem for multivariate normal
approximation on convex sets. The theorem is formulated in terms of a
multivariate extension of Stein couplings. We apply the results to a homogeneity
test in dense random graphs and to prove multivariate asymptotic normality
for certain doubly indexed permutation statistics.
\end{abstract}

%
\end{frontmatter}

\section{Introduction} \label{sec1}

Let $W$ and $Z$ be $d$-dimensional random vectors, $d\geq1$, where $Z$
has standard $d$-dimensional Gaussian distribution. We are concerned with
bounding the
quantity
%
\begin{equation}
\label{1} d_{c}\bigl(\law(W), \law(Z)\bigr) =\sup
_{A\in\mathcal{A}}\bigl\llvert \P(W\in A)-\P(Z \in A)\bigr\rrvert,
\end{equation}
where $\mathcal{A}$ denotes the collection of all the convex sets in
$\IR^d$.

Our main tool is Stein's method for the multivariate normal
distribution, which has already been used to obtain bounds on \eq{1},
the two
main contributions coming from G{\"o}tze \cite{Goetze1991} for sums of
independent random
vectors (see also Bhattacharya and Holmes \cite{Bhattacharya2010}), and Rinott and Rotar{$'$} \cite{Rinott1996} for
sums of
dependent random vectors that allow for a certain decomposition. Most other
contributions on multivariate normal approximation via Stein's method have
focused on smooth functions; see, for example, Barbour \cite{Barbour1990},
Goldstein and Rinott \cite{Goldstein1996}, Rai{\v{c}} \cite{Raic2004} and
Reinert and R{\"o}llin \cite{Reinert2009}.

The main aim of this article is to improve the results of Rinott and Rotar{$'$} \cite{Rinott1996} in
two important ways. First, we remove a logarithmic factor in the error
bound of Rinott and Rotar{$'$} \cite{Rinott1996}. The techniques that allow us to do this are
taken from Fang \cite{Fang2012} and will yield optimal rates of
convergence in some applications. Second, the
assumptions made on the dependence by Rinott and Rotar{$'$} \cite{Rinott1996} do not cover the
applications we will discuss here. Instead, we will use a multivariate
generalisation of Stein couplings to achieve the necessary generality. Stein
couplings, introduced by Chen and R{\"o}llin \cite{Chen2010b}, capture the minimal structural
assumption necessary to use Stein's method for normal approximation.

We will also keep the dependence of the
constants on the dimensionality explicit and as small as possible without
blowing up the proofs, but we do not pursue optimality in that respect.

The remainder of this article is organised as follows. In Section~\ref
{sec2}, we
will state our main abstract theorem, but we will postpone the (rather
technical) proof to Section~\ref{sec4}. In Section~\ref{sec3}, we
will discuss
two main applications, one involving permutation statistics and the
other a new
test for heterogeneity for dense graphs. In Section~\ref{sec5}, we
will present
some standard multivariate Stein couplings for reference.

\section{Main results} \label{sec2}

Stein couplings were introduced by Chen and R{\"o}llin \cite{Chen2010b} in order to unify
many of
the approaches developed around Stein's method for normal
approximation, such
as local approach, size biasing and exchangeable pairs, to name but a
few. In
the spirit of Chen and R{\"o}llin \cite{Chen2010b}, we give a multivariate definition of Stein
couplings.

\begin{definition}\label{def1}
A triple of square integrable $d$-dimensional random vectors $(W, W',
G)$ is
called a \emph{$d$-dimensional Stein coupling} if
%
\begin{equation}
\label{4} \E\bigl\{G^t F\bigl(W'\bigr)
-G^t F(W)\bigr\}=\E\bigl\{W^t F(W)\bigr\}
\end{equation}
for all $F\dvtx  \IR^d\rightarrow\IR^d$ for which the expectations exist.
\end{definition}

\begin{remark}\label{rem1} By choosing $F(w) = e_i$, where $e_i$ is
the $i$th
unit vector, it follows from \eq{4} that $\IE W_i = 0$. Therefore,
$\IE
W = 0$
is a necessary condition for a Stein coupling. Choosing $F(w) = w_j
e_i$, it
follows that
%
\begin{equation}
\label{106} \IE\bigl\{G\bigl(W'-W\bigr)^t\bigr\} =
\Cov(W).
\end{equation}
\end{remark}

Throughout this article, $\llvert  x\rrvert  $ denotes the
Euclidean norm of $x\in
\IR^d$,
and $\Id_d$ denotes the \mbox{$d$-dimensional} identity matrix. To shorten the
formulas somewhat, we will write $\IE^W( \cdot)$ to denote
conditional expectation $\IE( \cdot|W)$.

With this, we can formulate our main result.

\begin{theorem}\label{thm1} Let $(W,W',G)$ be a $d$-dimensional Stein
coupling. Assume that $\Cov(W)=\Id_d$. With $D=W'-W$,
suppose that there are positive constants $\alpha$ and $\beta$ such that
%
\begin{equation}
\label{5} \llvert G\rrvert \leq\alpha,\qquad\llvert D\rrvert \leq\beta.
\end{equation}
Then there is a universal constant $C$ such that
%
\begin{eqnarray}
\label{6} &&d_c\bigl(\law(W), \law(Z)\bigr)\nonumber
\\[-8pt]\\[-8pt]
&&\quad \leq C\bigl( d^{7/4} \alpha\E\llvert D\rrvert ^2+d^{1/4}
\beta +d^{7/8}\alpha^{1/2} B_1^{1/2}
+d^{3/8} B_2 +d^{1/8}B_3^{1/2}
\bigr),\nonumber
\end{eqnarray}
where $Z$ is a $d$-dimensional standard Gaussian random vector and
\begin{eqnarray*}
B_1 &=& \sqrt{\Var\E^W\llvert D\rrvert ^2},\qquad
B_2 = \sqrt{\sum_{i,j=1}^d
\Var\E^W (G_i D_j)},
\\
B_3 &=& \sqrt{\sum_{i,j,k=1}^d
\Var\E^W (G_i D_j D_k)}.
\end{eqnarray*}
\end{theorem}

As usual, we can upper bound $\Var\IE^W(\cdot)$ by
$\Var\IE^{\mathcal{F}}(\cdot)$ for any $\sigma$-algebra $\mathcal
{F}\supset
\sigma(W)$. This is a standard trick in Stein's method and will be used
in the
applications without further mention.

Note that, if $(W,W',G)$ is a $d$-dimensional Stein coupling
and $A$ is a $m\times d$ matrix, $m\geq1$, then $(AW,AW',AG)$ is an
$m$-dimensional Stein coupling. In this light, assuming that $\Cov(W) =
\Id_d$
is a matter of convenience rather than a real restriction. If $A$ is a
$d\times d$ matrix, denote by $\llVert  A\rrVert  _2$ its
operator norm with respect
to the Euclidean norm. Noticing that $d_c$ is invariant under linear
transformations, we
have the following consequence of Theorem~\ref{thm1}.

\begin{corollary}\label{cor1} Under the conditions of Theorem~\ref
{thm1}, but
now allowing
$\Cov(W) = \Sigma$ for any positive definite $\Sigma$, there is a
universal constant
$C$ such that
%
\begin{eqnarray}
\label{7} &&d_c\bigl(\law\bigl(\Sigma^{-1/2} W\bigr),
\law(Z)\bigr)\nonumber\\
&&\quad = d_c\bigl(\law( W), \law\bigl(\Sigma^{1/2}Z
\bigr)\bigr)
\\
&& \quad\leq C \bigl( d^{7/4} \alpha s_2^3\E\llvert D
\rrvert ^2+ d^{1/4} s_2\beta+ d^{7/8}
s_2^{3/2} \alpha^{1/2} B_1^{1/2}
+ d^{3/8} s_2^2 B_2
+d^{1/8} s_2^{3/2} B_3^{1/2}
\bigr),\nonumber
\end{eqnarray}
where $s_2 = \llVert \Sigma^{-1/2}\rrVert  _2$.
\end{corollary}

Note that the corollary cannot be expected to be informative if $\Sigma$
is singular or close to singular. In particular, the $W_i$ need to be
standardized so that $\Var W_i$, $1\leq i\leq d$, are all of the same
order. The proof of Corollary~\ref{cor1} is given in Section~\ref{sec4}.

\begin{remark}\label{rem2} If $(W,W')$ is an exchangeable pair of
$d$-dimensional vectors and
%
\begin{equation}
\label{8a} \IE^W\bigl(W'-W\bigr) = -\Lambda W
\end{equation}
for some invertible $d\times d$-matrix $\Lambda$, then
$(W,W',\frac{1}{2}\Lambda^{-1}(W'-W))$ is a Stein coupling and
Theorem~\ref{thm1} can be applied. In the special case where $\Lambda=
\lambda\II_d$, or in other words, if we have
%
\begin{equation}
\label{8} \IE^W\bigl(W'-W\bigr) = -\lambda W
\end{equation}
for some $0<\lambda<1$, then one can prove a special case of
Theorem~\ref{thm1} without using exchangeability, but only assuming
that $\law(W) = \law(W')$. A sketch of the proof will be given in
Section~\ref{sec4}.
This is analogous to Reinert and R{\"o}llin \cite{Reinert2009}, where a result similar to our
Theorem~\ref{thm1} was obtained for the special case of \eq{8a}, but
for a smooth metric, and where also exchangeabiliy was relaxed to equal
marginals in the special case of \eq{8}.
\end{remark}

\section{Applications} \label{sec3}

\subsection{A confidence interval for dense homogeneous random graphs}
\label{sec31}

One of the basic problems in the statistical analysis of graphs is to
test whether
the connections between vertices in a graph have arisen `completely at
random', or whether there is more structure in the graph. Among several
possible null hypotheses, one of the best-studied is the Erd\H{o}s--R\'
enyi random
graph $G(n,p)$, where two vertices are connected with probability $p$
and remain disconnected with probability $1-p$, independently of all else.\looseness=-1

Many test statistics have been analysed in the literature, such as
diameter, maximal degree, number of
triangles, etc.; see, for example, Pao, Coppersmith and Priebe \cite{Pao2011} for a recent
overview and simulation studies of the performance of these and other
test statistics. Despite the fact that much is known about the
behaviour of these test statistics under the null model $G(n,p)$, it
seems that little is known, at least theoretically, about how these
statistics behave under alternative models, such as heterogeneous
models, where the edge probabilities may vary. Here, as a first step,
we propose and justify a test that is based on the theory of dense
graph limits, and we will show that our test is \emph{consistent}, that
is, any deviation from the homogeneous model will eventually be
detected (in a sense made precise below).

\subsubsection*{Theory of dense graph limits}

Before we start with the statistical aspect of the problem, we first
give a brief introduction to the theory of dense graph limits. We will
only discuss those parts of the theory that are necessary for the
purpose of our application; we refer to Borgs, Chayes, Lov{\'a}sz, S{\'o}s and   Vesztergombi \cite{Borgs2008,Borgs2012} and
Bollob{\'a}s and Riordan \cite{Bollobas2009} for in-depth discussions. Also, dense graph limit
theory is intimately related to the theory of partially exchangeable
arrays as studied by Aldous \cite{Aldous1981}; see Diaconis and Janson \cite{Diaconis2008}.

In what follows, all graphs are assumed to be simple, that is, graphs
that contain no loops and no multiple edges, and, moreover, we assume
that all graphs are undirected. To\vadjust{\goodbreak} begin with, we consider non-random
graphs. Let $F$ and $G$ be graphs with $k$, respectively, $n$ vertices.
Denote by $\endo(F,G)$ the set of injective graph homomorphisms from
$F$ into $G$, and define
\[
t(F,G) = \frac{\llvert \endo(F,G)\rrvert  }{(n)_k},
\]
where $(n)_k := n(n-1)\cdots(n-k+1)$ (note that we follow the notation
of Bollob{\'a}s and Riordan \cite{Bollobas2009}; our $t$ is what Borgs, Chayes, Lov{\'a}sz, S{\'o}s and   Vesztergombi \cite{Borgs2008} denote
by $t_{\endo}$).
The number $\llvert \endo(F,G)\rrvert  $ is just the number
of copies of $F$ in
$G$ multiplied by the number of graph automorphisms of $F$. Since
$\llvert \endo(F,G)\rrvert  \leq(n)_k$, it is clear that
$0\leq t(F,G) \leq1$, and we
can think of this value as the ``density of $F$ in $G$''.

Let $(G_n)$ be a sequence of graphs (where $n\geq n_0$ for some
unspecified $n_0$) and for convenience assume that $G_n$ has $n$
vertices. We call this sequence a \emph{dense graph sequence} if the
number of
edges is of order $n^2$. In other words, if $K_m$ denotes the complete
graph on $m$ vertices, a graph sequence $(G_n)$ is called dense if
$\liminf_{n\toinf} t(K_2,G_n)>0$, and we will in fact mostly consider
sequences for which $\lim_{n\toinf} t(K_2,G_n)$ exists. Although it
would not pose any difficulties to allow the case $\lim_{n\toinf}
t(K_2,G_n) = 0$ (the ``sparse'' case), this only leads to degenerate
results in the context of dense graph theory, and is therefore excluded
for the sake of clarity.

We say that a dense graph sequence $(G_n)$ is convergent if $\lim_{n\toinf} t(F,G_n)$ exists for every finite graph $F$. We can
construct a metric $d$ on the set of isomorphism classes of finite
graphs, denoted by $\cF$, that quantifies this convergence. Let
$F_1,F_2,\dots$ be an arbitrary enumeration of the set of finite
graphs. For two graphs $G_1$ and $G_2$, let
\[
d(G_1,G_2) = \sum_{i\geq1}
2^{-i}\bigl\llvert t(F_i,G_1) -
t(F_i,G_2)\bigr\rrvert.
\]
It turns out that the metric space $(\cF,d)$ is not complete. The usual
way of constructing the completion of a metric space is to form
equivalence classes of sequences that are Cauchy with respect to the
metric. However, it turns out that there is a much more natural representation.

Let $\kappa\dvtx [0,1]^2\to[0,1]$ be a measurable and symmetric function;
we will
call any such function a \emph{standard kernel} (called \emph{graphon}
by Borgs, Chayes, Lov{\'a}sz, S{\'o}s and   Vesztergombi \cite{Borgs2008}). For any finite graph $F$ with
$k$ vertices, let
\[
t(F,\kappa) = \int_{0}^1\cdots\int
_0^1 \prod_{\{i,j\}\subset
E(F)}
\kappa(x_i,x_j) \,\mathrm{d} x_1\cdots\, \mathrm{d}
x_k,
\]
where $E(F)$ denotes the edge set of graph $F$. The quantity
$t(F,\kappa
)$ can be interpreted the ``density of $F$ in $\kappa$'', and we will
give a more intuitive representation of $t(F,\kappa)$ involving random
graphs later.

One of the key results of dense graph theory (see, for example, Borgs, Chayes, Lov{\'a}sz, S{\'o}s and   Vesztergombi \cite{Borgs2008}, Theorem~3.1) is the following. If $t(F,G_n)$ converges for
every $F$, that is, if $(G_n)$ is a Cauchy sequence with respect to
$d$, then there is a standard kernel $\kappa$ such that $\lim t(F,G_n)
= t(F,\kappa)$ for every $F$. We can therefore say that \emph{$\kappa$
is a limit of the graph sequence $(G_n)$}. Analogous to the fact that
there are graphs that are isomorphic to each other, there can (and
typically will) be several standard kernels representing the same
limit. Therefore, an additional step of forming equivalence classes of
standard kernels is necessary to obtain the actual completion of the
metric space $(\cF,d)$. Since we do not need this we refer again to
Borgs, Chayes, Lov{\'a}sz, S{\'o}s and   Vesztergombi \cite{Borgs2008,Borgs2012} on how to characterise these equivalence classes.

So far, all graphs have been non-random. If now $(G_n)$ is a sequence
of random graphs defined on a common probability space $\Omega$, we
will be interested in statements of the form ``$(G_n)$ converges to
$\kappa$ almost surely'', meaning that with probability $1$, the
realisation of a sequence $G_1(\omega),G_2(\omega),\dots$ converges to
$\kappa$ in the sense introduced above. Although it is possible to
allow for $\kappa$ to be random as well, we will only consider fixed
$\kappa$ in what follows.

For a given standard kernel $\kappa$, there is an elegant sampling
procedure to create random graphs that converge to $\kappa$ almost
surely. Let $U_1,U_2,\dots$ be a sequence of independent random
variables that are uniformly distributed on the interval $[0,1]$. To
construct $G_n$, connect vertices $i$ and $j$ with probability $\kappa
(U_i,U_j)$, independently of all other edges. We denote the
distribution of the graph $G_n$ obtained in this way by $G(n,\kappa)$
and it is clear that $G(n,p)$ for $0\leq p\leq1$ can be identified
with $G(n,\kappa)$ for $\kappa\equiv p$, the constant standard kernel.
Note that the edges of $G(n,\kappa)$ are conditionally independent
given $U_1,\dots,U_n$, but in general not unconditionally independent.
It is now easy to verify that,
if $G_n\sim G(n,\kappa)$, then
\[
\IE t(F,G_n) = t(F,\kappa).
\]
Furthermore, we have the following concentration result, which, by
Borell--Cantelli, immediately implies that $(G_n)$ converges to $\kappa$
almost surely.

\begin{lemma}[({Borgs, Chayes, Lov{\'a}sz, S{\'o}s and   Vesztergombi \cite{Borgs2008}, Lemma~4.4})]\label{lem20} If $G_n\sim
G(n,\kappa)$ for some standard kernel $\kappa$, and if $F$ is a graph
on $k$ vertices, then
\[
\IP\bigl[\bigl\llvert t(F,G_n) - t(F,\kappa)\bigr\rrvert > \eps
\bigr] \leq \exp\biggl(-\frac
{\eps^2n}{4k^2}\biggr)
\]
for every $\eps>0$.
\end{lemma}

\begin{remark}\label{rem30}
A remark about models that are more general than $G(n,\kappa)$ is in
place. It is important to note that dense graph theory is a first order
approximation of dense graphs, analogous to the law of large number for
random variables. It can be shown that the completion of $(\cF,d)$ is
compact and therefore, for any dense graph sequence, there must be
accumulation points which can be represented by a set $\cK$ of standard
kernels. So, if one considers graph models that produce dense graphs
that allow for more complex dependence between edges, any realisation
of a large enough graph from such a model will be close to at least one
of the standard kernels from its accumulation points $\cK$. Thus, from
this first order point of view, any dependence between the edges
becomes irrelevant in the limit, since every $\kappa\in\cK$ is also the
limit of the model $G(n,\kappa)$. As of yet, there seems to be no
established theory of second order fluctuations of dense graphs around
their limits that would capture more subtle aspects of such graphs.
\end{remark}

\subsubsection*{Characterisation of homogenous Erd\H{o}s--R\'enyi graphs}
Recall that $K_m$ denotes the complete graph of size $m$, and let $C_m$
be the cycle
graph of size $m$. Chung, Graham and Wilson \cite{Chung1989} proved the following surprising
result, which we shall present reformulated in the language of dense
graph limit theory (see Lov{\'a}sz and Szegedy \cite{Lovasz2011} for generalisations of these
findings).

\begin{theorem}[({Chung, Graham and Wilson \cite{Chung1989}, Theorem~1})]\label{thm31} If $(G_n)$
is a (non-random) dense graph sequence such that
\[
t(K_2,G_n) \to p\quad\mbox{and}\quad t(C_4,G_n)
\to p^4
\]
for some $0<p\leq1$, then $(G_n)$ converges and the limit is the
constant standard kernel $\kappa\equiv p$.
\end{theorem}

In other words, $\kappa\equiv p$ is the only standard kernel with
$t(K_2,\kappa) =
p$ and $t(C_4,\kappa)= p^4$, and it is not difficult to show that, if
$\kappa$ is not constant and $t(K_2,\kappa)=p$, then $t(C_4,\kappa) >
p^4$. This result suggests that we
can use the number of edges and $4$-cycles in order to test whether
$\kappa$ is constant or not.
Indeed, for $G_n\sim G(n,\kappa)$ with non-constant $\kappa$, we should
be able to detect a discrepancy between the edge density to the fourth
power and 4-cycle density if $n$ is large enough.

However, some care is needed. If $G_n$ is a given graph of size $n$,
define the two
statistics
\[
T_1(G_n) = \frac{\llvert \endo(K_2,G_n)\rrvert  }{2},\qquad T_2(G_n)
= \frac{\llvert \endo(C_4,G_n)\rrvert  }{8}.
\]
The factors $2$ and $8$, respectively, are the sizes of the automorphism
groups of $K_2$ and $C_4$, respectively. Therefore, $T_1$ is the number of
edges in $G_n$ and $T_2$ is the number of $4$-cycles in $G_n$.
By straightforward calculations we have that, if $G_n\sim G(n,p)$,
\[
\Var\bigl(T_1(G_n)\bigr)={n \choose2} p(1-p),\qquad\Cov
\bigl(T_1(G_n), T_2(G_n)
\bigr)=12{n\choose4} p^4(1-p)
\]
and
\begin{eqnarray*}
\Var\bigl(T_2(G_n)\bigr) = 3{n\choose4}
p^4 (1-p) \bigl(1+p-13p^2+4np^2+35p^3-24np^3+4n^2p^3
\bigr).
\end{eqnarray*}
It is clear from this that $\Cor(T_1(G_n),T_2(G_n))\to1$ as $n\toinf$,
hence, in the limit, the fluctuation of the   number of $4$-cycles
is determined by that of the number of edges in the graph; see Janson and Nowicki \cite{Janson1991} for such and more general results. Thus, we cannot use
these values directly to construct our test.

Following Janson and Nowicki \cite{Janson1991}, we can instead consider the density of $4$-cycles
\emph{corrected} by the edge density (this is essentially the first non-leading
term in a Hoeffding-type decomposition for the $4$-cycle
count). To this end, define the normalised edge count
\[
W_1(p,G_n) = \frac{T_1(G_n) - {n\choose2}p }{\sigma_1}, \qquad\mbox{with }
\sigma_1^2 = {n\choose2}p(1-p),
\]
and the \emph{corrected} and normalised $4$-cycle count
\[
W_2(p,G_n) = \frac{T_2(G_n) - 2{n-2 \choose2}p^3T_1(G_n) + 9{n\choose
4}p^4}{\sigma_2}
\]
with
\[
\sigma_2^2 = 3{n\choose4}p^4(1-p)^2
\bigl(1+2p+(4n-11)p^2\bigr);
\]
it is easy to see that $\Cov(W_1,W_2)=0$.
In order to motivate the choice of $W_1$ and $W_2$, note that, from
Lemma~\ref{lem20} and for general $\kappa$ and $G_n\sim G(n,\kappa)$,
\begin{eqnarray*}
\frac{W_1(p,G_n)}{n} & \to&\frac{1}{\sqrt{2p(1-p)}}\bigl(t(K_2,\kappa)-p\bigr),
\\
\frac{W_2(p,G_n)}{n^{3/2}} & \to&\frac{1}{4\sqrt{2} p^3(1-p)} \bigl(t(C_4,\kappa
)-4p^3t(K_2,\kappa)+3p^2\bigr)
\end{eqnarray*}
almost surely as $n\toinf$, so that $W_1(p,G_n)$ and $W_2(p,G_n)$ can
only expected to be near zero if $\kappa\equiv p$.

Barbour, Karo{\'n}ski and Ruci{\'n}ski \cite{Barbour1989} use Stein's method to prove univariate normal
approximations of subgraph counts and related statistics, but for
quantities such as $W_2$ they resort to the method of
moments. Corresponding multivariate results where obtained by Janson and Nowicki \cite{Janson1991} in great
generality for incomplete $U$-statistics using Hoeffding-type
decompositions and the methods of moments. For degenerate statistics
like $W_2$ they state that ``Stein's method does not seem to work in
that case''.

The reason that $W_2$ is more difficult to handle is that, if
represented as an incomplete $U$-statistic, many of the summands are
uncorrelated (see \eq{10} below), which requires more delicate
estimates. We note that the arguments of Barbour, Karo{\'n}ski and Ruci{\'n}ski \cite{Barbour1989} could be,
in fact, improved to cover such cases as well.

\begin{theorem}\label{thm2} Let $G_n\sim G(n,p)$ be a realisation of an
Erd\H{o}s--R\'enyi random graph on $n$ vertices with edge probability
$p$. Let $W =
(W_1(p,G_n),W_2(p,G_n))$ and
let $Z$ be a standard bi-variate normal random variable. There is a
universal constant $C$ independent of $p$ and $n$ such that
\[
d_c\bigl(\law(W), \law(Z)\bigr)\leq\frac{C}{p^9(1-p)^3\sqrt{n}}.
\]
\end{theorem}

Theorem~\ref{thm2} justifies the following procedure to construct a confidence
set for the family of Erd\H{o}s--R\'enyi random graphs. Let $G_n$ be a
simple graph
of size $n$.
Fix $0<\alpha<1$ and define the $1-\alpha$
confidence set as
\[
\CI_{1-\alpha}(G_n) = \bigl\{0<p<1 \dvtx  W_1^2(p,G_n)+W_2^2(p,G_n)
\leq q_{1-\alpha}\bigr\},
\]
where $q_{1-\alpha}$ is the $1-\alpha$ quantile of the $\chi^2$-distribution
with $2$ degrees of freedom. In words, $\CI_{1-\alpha}(G_n)$ is the set
of those $p$ for
which $G_n$ is ``compatible'' with the model $G(n,p)$ at the
significance level $\alpha$. If $\CI_{1-\alpha}(G_n)$ is empty, then
$G_n$ is not compatible with any homogeneous Erd\H{o}s--R\'enyi random graph
model.

For what follows, denote by $\IP_\kappa$ the distribution of $G_n$
under the law $G(n,\kappa)$ and let $\IP_p = \IP_\kappa$ for
$\kappa
\equiv p$.

\begin{corollary} For any given $0<p_l < p_u < 1$,
\[
\IP_p\bigl[p\notin\CI_{1-\alpha}(G_n)\bigr] =
\alpha + \bigo\bigl(n^{-1/2}\bigr)
\]
uniformly in $p\in[p_l,p_u]$ as $n\toinf$. Furthermore, if $\kappa$ is
a non-constant standard kernel, and if $n\geq\max\{19,54q_{1-\alpha
}^{1/2}/r_\kappa\}$, we have
%
\begin{equation}
\label{9} \IP_\kappa\bigl[\CI_{1-\alpha}(G_n)\neq
\emptyset\bigr] \leq2\exp\biggl(-\frac
{nr_\kappa^2}{144}\biggr),
\end{equation}
where
\[
r_\kappa^2 = \inf_{0<p<1}\bigl\{
\bigl(t(K_2,\kappa)-t(K_2,p)\bigr)^2 +
\bigl(t(C_4,\kappa)-t(C_4,p)\bigr)^2\bigr\}
\]
(note that $r_\kappa>0$ from Theorem~\ref{thm31} and the discussion
thereafter).
\end{corollary}

\begin{pf} The first part is immediate from Theorem~\ref{thm2}. For the
second part assume that $\kappa$ is not constant. Consider the points
$b(\kappa) = (t(K_2,\kappa),t(C_4, \kappa))$ and $b(p) = (p,p^4)$, and,
by slight abuse of notation, $b(n) = (t(K_2,G_n),t(C_4, G_n))$. Using
Lemma~\ref{lem20}, we have
%
\begin{eqnarray}
\label{306} &&\IP_\kappa\bigl[\bigl\llvert b(n) - b(\kappa)\bigr
\rrvert > \eps\bigr]\nonumber\\
&&\quad\leq\IP_\kappa\bigl[ \bigl\llvert b_1(n)-b_1(
\kappa)\bigr\rrvert > \eps /\sqrt{2}\bigr] + \IP_\kappa\bigl[ \bigl
\llvert b_2(n)-b_2(\kappa)\bigr\rrvert > \eps /\sqrt{2}
\bigr]
\\
&&\quad\leq2\exp\bigl(-\eps^2n/128\bigr)\nonumber
\end{eqnarray}
for any $\eps>0$.
Now, note that we can write
\[
W_1(p,G_n) = \frac{{n \choose2}}{\sigma_1} \bigl(b_1(n)-b_1(p)
\bigr),\qquad W_2(p,G_n)=\frac{3 {n \choose4}}{\sigma_2} w(n,p),
\]
where
\[
w(n,p) = \bigl(b_2(n)-b_2(p)\bigr) - 4p^3
\bigl(b_1(n)-b_1(p)\bigr).
\]
Let
\[
\delta= \frac{\sqrt{67}-4}{306}r_\kappa
\]
and define the events
\begin{eqnarray*}
A_1(p) &=& \bigl\{\bigl\llvert b(n)-b(\kappa)\bigr\rrvert
\leq8r_\kappa /9,\bigl\llvert b_1(n)-b_1(p)\bigr
\rrvert > \delta\bigr\},
\\
A_2(p) &=& \bigl\{\bigl\llvert b(n)-b(\kappa)\bigr\rrvert
\leq8r_\kappa /9,\bigl\llvert b_1(n)-b_1(p)\bigr
\rrvert \leq\delta\bigr\}.
\end{eqnarray*}
On one hand, we have
%
\begin{equation}
\label{302} W_1(p,G_n)^2 \geq\I
\bigl[A_1(p)\bigr]\biggl(\frac{{n\choose2}\delta
}{\sigma_1}\biggr)^2.
\end{equation}
On the other hand, since $\llvert  b(n)-b(\kappa)\rrvert  \leq
8r_\kappa/9$ implies
$\llvert  b(n)-b(p)\rrvert  \geq r_\kappa/9$, we have
\[
w(n,p) > \sqrt{(r_\kappa/9)^2 -\delta^2} - 4
\delta= \frac{r_\kappa}{18}
\]
on $A_2(p)$, and hence
%
\begin{equation}
\label{301} W_2(p,G_n)^2 \geq\I
\bigl[A_2(p)\bigr]\biggl(\frac{3{n\choose4}r_\kappa
}{18\sigma_2}\biggr)^2.
\end{equation}
Setting $A(p) = A_1(p)\cup A_2(p)$ and putting \eq{302} and \eq{301}
together, we obtain
\begin{eqnarray*}
&&W_1(p,G_n)^2+W_2(p,G_n)^2
\\
&&\quad\geq \I\bigl[A(p)\bigr] r_\kappa^2 \min\biggl\{ \biggl(
\frac{{n\choose4}}{6\sigma_2}\biggr)^2, \biggl(\frac{(\sqrt{67}-4){n\choose2}}{306\sigma_1}
\biggr)^2 \biggr\}
\\
&&\quad \geq \I\bigl[A(p)\bigr] r_\kappa^2 \min\bigl\{ 4.3
\cdot10^{-3}(n-1)_3, 3.7\cdot10^{-4}
(n)_2 \bigr\}.
\end{eqnarray*}
If $n\geq19$, we have
\[
\min\bigl\{ 4.3\cdot10^{-3}(n-1)_3, 3.7
\cdot10^{-4} (n)_2 \bigr\} \geq 3.5\cdot10^{-4}
n^2.
\]
Hence, if $n\geq\max\{19,(q_{1-\alpha}/( 3.5\cdot10^{-4}r_\kappa
^2))^{1/2}\}$, and using \eq{306},
\[
\IP_{\kappa}[\CI_{1-\alpha}=\emptyset] \geq\IP_\kappa\bigl[
\bigl\llvert b(n)-b(k)\bigr\rrvert \leq8r_\kappa/9\bigr] \geq1-2\exp
\bigl(-nr_\kappa^2/144\bigr),
\]
which implies \eq{9}.
\end{pf}

\begin{remark} Note that \eq{9} essentially says that, if the true
standard kernel $\kappa$ is non-constant, the test will eventually
detect this for $n$ large enough. It is not clear if this is still true if
$4$-cycles were to be replaced by triangles. Chung, Graham and Wilson \cite{Chung1989}, page 361,
give an example of non-constant standard kernel $\kappa$ for which
\[
t(K_2,\kappa) = \tfrac{1}{2},\qquad t(C_3,\kappa) =
\tfrac{1}{8},
\]
which also holds for the constant standard kernel $\kappa\equiv1/2$.
\end{remark}

\begin{remark}\label{rem33} To go back to the question posed at the
beginning of the section, namely to decide whether a given graph $G_n$
is compatible with any homogenous model $G(n,p)$, $0<p < 1$, we can
formulate this now more precisely as the testing problem
\[
H_0 \dvtx  G_n \sim G(n,p)\qquad\mbox{for some }0< p < 1
\]
against
\[
H_1 \dvtx  G_n\sim G(n,\kappa)\qquad \mbox{with }\kappa\not
\equiv p\mbox{ for all }0< p< 1.
\]
As we have already pointed out in Remark~\ref{rem30}, from the point of
view of first order approximation of dense graph limit theory, the
alternative hypothesis is already in its most general form, since the
models $G(n,\kappa)$ cover all possible dense graph limits.

We can now define a test $\psi(G_n)$ that rejects the null hypothesis
if $C_{1-\alpha}(G_n)$ is empty, that is,
\begin{eqnarray*}
\psi(G_n) = \I\bigl[W_1^2(p,G_n)+W_2^2(p,G_n)
> q_{1-\alpha}\mbox{ for all }0<p<1\bigr].
\end{eqnarray*}
Since
\[
\IP_p\bigl[\psi(G_n) = 1\bigr] \leq\IP_p
\bigl[W_1^2(p,G_n)+W_2^2(p,G_n)
> q_{1-\alpha}\bigr] = \alpha+ \bigo\bigl(n^{-1/2}\bigr),
\]
this test has an asymptotic significance level of $\alpha$ or less.
Whether the asymptotic significance level is strictly less than or
equal to $\alpha$ depends on the asymptotic behaviour of the quantity
\[
\inf_{0<p<1}\bigl\{W_1^2(p,G_n)+W_2^2(p,G_n)
\bigr\},
\]
which cannot be expected to have a $\chi^2$-distribution. Numerical
simulations indicate that the asymptotic significance level of $\psi$
is strictly less than $\alpha$, but a mathematical proof of this
observation eludes us.
\end{remark}

Before we prove Theorem~\ref{thm2}, we need some notation and technical
lemmas. For the remainder of this subsection, that is until the end of
the proof of Theorem~\ref{thm2}, we will follow the convention that the
elements in an ordered $m$-tuple $(i_1,\dots,i_m)$ of integers are
pairwise different and range from $1$ to $n$, and we will assume the
same for sets written as $\{i_1,\dots,i_m\}$, so that $\llvert \{
i_1,\dots,i_m\}\rrvert  = m$ always. For every $(i,j,k,l)$ let
\[
\eta_{ijkl}=I_{ij} I_{jk} I_{kl}
I_{il}-p^3( I_{ij}+I_{jk}+I_{kl}+
I_{il})+3p^4,
\]
where $I_{ij} = I_{ji}$ is the indicator of the event that there is an edge
connecting $i$ and $j$. Note that between every set of four vertices $\{
i,j,k,l\}$, only three essentially different $4$-cycles can be spanned,
so that, for example, the set of eight $4$-tuples
\begin{eqnarray*}
&& \bigl\{(i,j,k,l),(j,k,l,i),(k,l,i,j),(l,i,j,k),
\\
&&\hphantom{\bigl\{}(i,l,k,j),(l,k,j,i),(k,j,i,l),(j,i,l,k)\bigr\}
\end{eqnarray*}
represent the same $4$-cycle, and hence
\begin{eqnarray*}
 \eta_{ijkl} &=& \eta_{jkli}= \eta_{klij}=
\eta_{lijk}
\\
& =& \eta_{ilkj} = \eta_{lkji} = \eta_{kjil} =
\eta_{jilk}.
\end{eqnarray*}
It is also straightforward to verify that, if $\cV\subset\{1,\dots,n\}$
is of arbitrary size, then
%
\begin{equation}
\label{10} \E\bigl\{\eta_{ijkl}|(I_{uv})_{u,v\in\cV}
\bigr\}=0
\end{equation}
for any $(i,j,k,l)$ with $\llvert \{i,j,k,l\}\cap\cV\rrvert
\leq2$. From \eq
{10}, we can easily deduce statements about mixed moments. For example,
for any $(i,j,k,l)$ and any $(u,v)$, we have
\[
\E\{\eta_{ijkl} I_{uv}\}=0,
\]
or, if $|\{i,j,k,l\} \cap\{u,v,w,m\}|\leq2$, we have
\[
\E\{\eta_{ijkl} \eta_{uvwm}\}=0.
\]
Whenever we will be using such identities (or similar identities with
more factors) in the proof, we will only refer to \eq{10}, since
obtaining these covariance formulas from \eq{10} is straightforward.

For each $\nu= \{i,j,k,l\}$, let
\begin{eqnarray*}
X_{1,\nu} & =& \frac{1}{{n-2\choose2}\sigma
_1}(I_{ij}+I_{ik}+I_{il}+I_{jk}+I_{jl}+I_{kl}
- 6p),
\\
X_{2,\nu} & = &\frac{1}{\sigma_2}(\eta_{ijkl} +
\eta_{ijlk} + \eta_{ikjl}),
\end{eqnarray*}
and $X_\nu=(X_{1,\nu}, X_{2,\nu})^t$. Now we can represent $W$ as a
sum of
locally dependent random vectors, namely
%
\begin{equation}
\label{12} W = \sum_{\nu} X_{\nu},
\end{equation}
where the sum ranges over all subsets $\nu=\{i,j,k,l\}$. To see
that \eq{12} is the same as in Theorem~\ref{thm2}, recall that between
each set of four vertices $\{i,j,k,l\}$, there can be at most three
different $4$-cycles, and that, in the definition of $X_{2,\nu}$, one
representative of each of them is picked. Furthermore, each edge
$I_{ij}$ is
over-counted ${n-2\choose2}$ times, hence the additional factor ${n-2
\choose2}^{-1}$ in the definition of $X_{1,\nu}$. It is
straightforward to check that
\[
\E W=0,\qquad\IE\bigl\{WW^t\bigr\} = \Id_2,
\]
where $\Id_m$ is the $m$-dimensional unit matrix. Note that $X_\nu$ and
$X_\xi$ are independent whenever $\llvert \nu\cap\xi\rrvert  \leq1$, that is,
share at
most one vertex. Hence, for each $\nu$, we define the set
$A_\nu:=\{\xi\dvtx  |\nu\cap\xi|\geq2\}$, the `neighbourhood'
of $X_\nu
$. For given
$\nu$, we then have that the collection $(X_\xi)_{\xi\notin A_\nu}$ is
independent of
$X_\nu$. Therefore, if $I$ is uniformly distributed over all $\nu$,
%
\begin{equation}
\label{13} \bigl(W,W',G\bigr):=\biggl(W, W-\sum
_{\nu\in A_I} X_\nu, -{n \choose4}X_I
\biggr)
\end{equation}
is a Stein coupling (cf. Section~\ref{sec5}).

Since the sequence $(G_n)$ starts at some unspecified integer $n_0$, we
can assume without loss of generality that $n_0\geq3$, and, hence, use
$G_1$ and $G_2$ to denote the first, respectively, second component of
the vector $G$, rather than elements from the random graph sequence $(G_n)$.

\begin{pf*}{Proof of Theorem~\ref{thm2}}
We apply Theorem~\ref{thm1} for the Stein coupling given in \eq{13}.
Let as usual
$D=W'-W$. In what follows, $C$ denotes a positive constant independent of
$p$ and $n$, possibly different from line to line. Note first that
\[
\sigma_1^2\geq Cn^2p(1-p),\qquad
\sigma_2^2 \geq C n^5p^6(1-p)^2.
\]
Hence,
\[
\llvert X_\nu\rrvert \leq C\biggl(\frac{1}{n^2\sigma_1}+
\frac{1}{\sigma
_2}\biggr)\leq \frac{C}{n^{5/2}p^3(1-p)}
\]
and $\llvert A_\nu\rrvert \leq Cn^2$, which yields the upper bounds
%
\begin{equation}
\label{14} \llvert G\rrvert \leq\frac{Cn^{3/2}}{p^3(1-p)} =: \alpha,\qquad \llvert D
\rrvert \leq\frac{C}{p^3(1-p)n^{1/2}} =: \beta.
\end{equation}
The second moment of $\llvert D\rrvert $ can be calculated as follows. Noting that
$\llvert \xi\cap\xi'\rrvert \leq1$ implies $\IE(X_{1,\xi} X_{1,\xi'})=0$, and
$\llvert \xi\cap\xi'\rrvert \leq2$ implies $\IE(X_{2,\xi} X_{2,\xi'})=0$, which
follows from \eq{10}, we have
%
\begin{eqnarray}
\label{15} \E\llvert D\rrvert ^2 & =& \E D_1^2+
\E D_2^2\nonumber
\\
& =& \frac{1}{{n \choose4}}\sum_{\nu}\sum
_{\xi,\xi'\in
A_\nu}\IE(X_{1,\xi}X_{1,\xi'}) +
\frac{1}{{n \choose4}}\sum_{\nu}\sum
_{\xi,\xi'\in
A_\nu}\IE(X_{2,\xi} X_{2,\xi'})\nonumber\\[-8pt]\\[-8pt]
& \leq&\frac{C}{n^4}\times n^4 \times n^2 \times
n^2 \times\frac
{1}{n^4 \sigma_1^2} + \frac{C}{n^4} \times
n^4 \times n^2 \times n\times \frac{1}{\sigma_2^2}\nonumber
\\
& \leq&\frac{C}{n^2 p^6 (1-p)^2}.\nonumber
\end{eqnarray}
Define the $\sigma$-field $\mathcal{F} = \sigma(G_n)$. Clearly,
$\mathcal{F}
\supset\sigma(W)$. In the following, we calculate the variances of the
conditional expectations in the bound \eq{6}. First,
%
\begin{eqnarray}
\label{200} \Var\bigl(\E^\mathcal{F} G_1 D_1
\bigr) & \leq&\Var\biggl(\sum_\nu X_{1,\nu}
\sum_{\xi\in A_\nu} X_{1,\nu}\biggr)\nonumber
\\[-8pt]\\[-8pt]
& =& \sum_{\nu, \nu'} \sum_{\xi\in A_\nu, \xi'\in
A_{\nu'}}
\Cov(X_{1,\nu}X_{1,\xi}, X_{1,\nu'} X_{1,\xi'}) \leq
\frac{Cn^{10}}{n^8 \sigma_1^4},\nonumber
\end{eqnarray}
where the last inequality follows from the fact that $\Cov(X_{1,\nu
}X_{1,\xi},
X_{1,\nu'} X_{1,\xi'})\neq0$ can occur only if $\llvert (\nu
\cup\xi
)\cap
(\nu'\cup\xi')\rrvert  \geq2$. By the same argument,
%
\begin{equation}
\label{201} \Var\bigl(\E^\mathcal{F} G_1 D_2
\bigr) \leq\sum_{\nu, \nu'} \sum
_{\xi\in A_\nu,
\xi'\in A_{\nu'}}\Cov(X_{1,\nu}X_{2,\xi},
X_{1,\nu'} X_{2,\xi'}) \leq\frac{Cn^{10}}{n^4 \sigma_1^2 \sigma_2^2}
\end{equation}
and
%
\begin{equation}
\label{202} \Var\bigl(\E^\mathcal{F} G_2 D_1
\bigr)\leq\frac{Cn^{10}}{n^4 \sigma_1^2
\sigma_2^2}.
\end{equation}
In order to bound $\Var(\E^\mathcal{F} G_2 D_2)$, we argue that
%
\begin{equation}
\label{123} \Cov(X_{2,\nu}X_{2,\xi}, X_{2,\nu'}
X_{2,\xi'}) \neq0 \qquad\mbox{implies }\bigl\llvert \nu\cup\xi\cup
\nu'\cup\xi'\bigr\rrvert \leq9,
\end{equation}
from which we can deduce that
%
\begin{equation}
\label{203} \Var\bigl(\E^\mathcal{F} G_2 D_2
\bigr)\leq\frac{Cn^{9}}{\sigma_2^4}.
\end{equation}
To show \eq{123}, note that the left-hand side implies that
\begin{enumerate}[(ii)]
\item[(i)] any intersection of $\nu$, $\xi$, $\nu'$ or $\xi'$ with the
union of the other three sets has at least three elements (otherwise we
would obtain a contradiction with \eq{10}), and
\item[(ii)] at least one of the intersections $\nu\cap\nu'$, $\nu
\cap\xi
'$, $\xi\cap\nu'$ and $\xi\cap\xi'$ has at least two elements
(otherwise $X_{1,\nu'}X_{1,\xi'}$ and $X_{1,\nu'}X_{1,\xi'}$ would be
independent).
\end{enumerate}
Assume now that the left-hand side of \eq{123} is true. Since $\xi\in
A_\nu$, we have $\llvert \nu\cap\xi\rrvert  \geq2$, and
hence $\llvert \nu
\cup\xi
\rrvert  \leq6$, and similarly $\llvert \nu'\cup\xi'\rrvert  \leq6$. Using (ii), we
deduce that one of the three inequalities $\llvert (\nu\cup\xi
)\cap
\nu
'\rrvert  \geq2$, $\llvert (\nu\cup\xi)\cap\xi'\rrvert  \geq2$, or $\llvert (\nu
'\cup\xi
')\cap\nu\rrvert  \geq2$ must hold. If the first inequality
holds, we obtain
$\llvert \nu\cup\xi\cup\nu'\rrvert  \leq8$, and, using
(i), that $\llvert \nu\cup
\xi\cup\nu'\cup\xi'\rrvert  \leq9$; the other two
inequalities are analogous.
This concludes the proof of \eq{123}.

Collecting the bounds \eq{200}, \eq{201}, \eq{202} and \eq{203}, we obtain
\[
B_2\leq\frac{C}{n^{1/2}p^{6}(1-p)^2}.
\]
By similar arguments,
\[
\Var\E^\mathcal{F} \bigl(D_1^2\bigr)\leq
\frac{C}{n^2 \sigma_1^4},\qquad \Var\E^\mathcal{F} \bigl(D_2^2
\bigr)\leq\frac{Cn^5}{\sigma_2^4},
\]
and, hence,
\[
B_1\leq\frac{C}{n^{5/2}p^{6}(1-p)^2}.
\]

The following bounds can be obtained in a similar fashion, again
using \eq{10}, but we omit the tedious details. We have
\begin{eqnarray*}
\Var\E^\mathcal{F} \bigl(G_1 D_1^2
\bigr)&\leq&\frac{Cn^{14}}{n^{12} \sigma
_1^6}, \qquad\Var\E^\mathcal{F} (G_1
D_1 D_2)\leq\frac{Cn^{14}}{n^8 \sigma_1^4 \sigma_2^2},
\\
\Var\E^\mathcal{F} \bigl(G_1 D_2^2
\bigr)&\leq&\frac{Cn^{13}}{n^4 \sigma_1^2
\sigma_2^4},\qquad \Var\E^\mathcal{F} \bigl(G_2
D_1^2\bigr)\leq\frac{Cn^{14}}{n^8 \sigma_1^4
\sigma_2^2},
\\
\Var\E^\mathcal{F} (G_2 D_1 D_2)&\leq&
\frac{Cn^{13}}{n^4 \sigma_1^2
\sigma_2^4}, \qquad\Var\E^\mathcal{F} \bigl(G_2
D_2^2\bigr)\leq\frac{Cn^{13}}{\sigma_2^6},
\end{eqnarray*}
and therefore
\[
B_3 \leq\frac{C}{np^9(1-p)^{3}}.
\]
Collecting the bounds on $B_1$, $B_2$ and $B_3$, in combination with
\eq{14} and \eq{15}, yields the final estimate via \eq{6}.
\end{pf*}

\subsection{Joint normality of certain permutation statistics}

Let $M$ be a real $n\times n$ matrix and assume that $M$ is anti-symmetric,
that is, for each $u,v\in\{1,\dots,n\}$, we have
\[
M_{uv} = - M_{vu}.
\]
Note that $M_{uu}=0$.
Let $\pi$ be a permutation of size $n$, chosen uniformly at random, and
consider the statistic
%
\begin{equation}
\label{16} W = \sum_{i<j} M_{\pi(i)\pi(j)}.
\end{equation}
Here, sums of the form $\sum_{i<j}$ have to be interpreted as double
sums $\sum_{i=1}^{n-1}\sum_{j=i+1}^n$. If it is to be interpreted as a
single sum, we will explicitly state the summation index using the
notation $\sum_{j:i<j}$.\vspace*{2pt}

Permutation statistics of the form \eq{16} were considered by Fulman \cite{Fulman2004} and they
are a special case of doubly-indexed permutation statistics
%
\begin{equation}
\label{17} \sum_{i,j} a\bigl(i,j,\pi(i),\pi(j)
\bigr)
\end{equation}
with
\[
a(i,j,u,v) = \I[i<j] M_{uv}.
\]
The reason to study \eq{16} is that two important properties of permutations,
the number of descents and inversions, can be readily represented in
this form.
Choosing $M_{u,u+1}=-1$ and $M_{uv} = 0$ for all other $v>u$ (for
$v<u$, $M_{uv}$ is defined via anti-symmetry), \eq{16} becomes
$2\Des(\pi^{-1})-(n-1)$, where $\Des(\pi)$ is the number of descents
of $\pi$;
with $M_{uv} = -1$ for all $u<v$, \eq{16}
becomes $2\Inv(\pi^{-1}) - {n\choose2}$, where $\Inv(\pi)$ is the
number of
inversions of $\pi$.

Using Stein's method, Zhao, Bai, Chao and Liang \cite{Zhao1997} prove a general Berry--Esseen
type theorem
for sums of the form \eq{17}, but their results do not apply to the
number of
descents $\Des(\pi)$, which seems to be ``too sparse''. In contrast,
using a
special exchangeable pair, Fulman \cite{Fulman2004} was able to obtain a rate of
convergence of $n^{-1/2}$ for the Kolmogorov metric for both, the
number of
descents and inversions.

We shall extend Fulman's results to the multivariate setting. Furthermore,
we are able to remove a certain condition on $M$ (present in Fulman's work),
arising from the requirement of exchangeability; cf. Remark~\ref{rem2}.
In addition to extending the exchangeable pair approach by Fulman \cite{Fulman2004}, we also provide a result using the local approach.

Let $M^{(1)},\dots,M^{(d)}$ be a sequence of real $n\times n$ matrices
and assume that
each matrix is anti-symmetric.
For each $r$, define $W_r=\sum_{i<j} M^{(r)}_{\pi(i) \pi(j)}$. As in
Fulman \cite{Fulman2004}, define
\[
A_u^{(r)}=\sum_{v:v>u}
M^{(r)}_{uv},\qquad B^{(r)}_u=\sum
_{v:v<u} M^{(r)}_{vu}.
\]
The mean and
covariances of $W=(W_1,\dots,W_d)$ are given in the following lemma.

\begin{lemma}\label{lem2}
We have $\E W=0$ and
%
\begin{equation}
\label{108} \Cov(W_r, W_s)=\frac{1}{3}\biggl(
\sum_{u<v} M^{(r)}_{uv}
M^{(s)}_{uv} +\sum_{u}
\bigl(A^{(r)}_u-B^{(r)}_u\bigr)
\bigl(A^{(s)}_u-B^{(s)}_u\bigr)\biggr).
\end{equation}
\end{lemma}

\begin{pf}
Both the covariance and the right-hand side of \eq{108} are symmetric
bilinear forms on the vector space of all anti-symmetric matrices.
Moreover, by Lemma~4.3.1 of Fulman \cite{Fulman2004}, both expressions match
for $M^{(r)}=M^{(s)}$. Since a symmetric bilinear form is uniquely
determined by the corresponding quadratic form, the results follows.
\end{pf}

With $W=(W_1,\dots, W_d)^t$, we have the following result.

\begin{theorem}\label{thm3}
Let $W$ be as above and let
%
\begin{equation}
\label{18} \beta=\sup_{r, u} \sum
_{v} \bigl\llvert M^{(r)}_{uv}\bigr
\rrvert,\qquad\beta_2=\sup_{r, u} \sum
_{v} \bigl( M^{(r)}_{uv}
\bigr)^2.
\end{equation}
Assume $\Var(W_r)=1$ for each $1\leq r\leq d$. Then, with $\Sigma
=\Cov
(W)$, there is a positive constant $C_d$ depending only
on $d$, such that
%
\begin{equation}
\label{109} d_c\bigl(\law( W), \law\bigl(\Sigma^{1/2} Z
\bigr)\bigr)\leq C_d \bigl\llVert \Sigma^{-1/2}\bigr\rrVert
_2^2 \bigl(n \beta ^3+n^{1/2}\beta
\beta _2^{1/4} \bigr).
\end{equation}
\end{theorem}

Although Theorem~\ref{thm3} is widely applicable, it does not yield
optimal bounds for the applications discussed below. To this end, we
also give a theorem that gives better bounds under the more specific
situation where the non-zero entries of $M^{(r)}$ are all near the
diagonal and $W_1$ is the normalised number of inversions.

\begin{theorem}\label{thm37} Assume the situation of Theorem~\ref
{thm3}. In addition, assume that $W_1$ is of the specific form
\[
W_1=\frac{\Inv(\pi)-\fracd{1}{2}{n \choose2}}{\sqrt{\fracb{n(n-1)(2n+5)}{72}}},
\]
where $\Inv(\pi)$ is the number of inversions of $\pi$. Assume
further that there is a positive integer $m$ such that
\[
M^{(r)}_{uv}=0,\qquad\mbox{if $\llvert u-v\rrvert >m$ and $2\leq
r\leq d$.}
\]
Then
%
\begin{equation}
\label{111} d_c\bigl(\law( W), \law\bigl(\Sigma^{1/2} Z
\bigr)\bigr)\leq C_{d,m} \bigl\llVert \Sigma^{-1/2}\bigr\rrVert
_2^2 n \beta^3,
\end{equation}
where $C_{d,m}$ is a positive constant depending only on $d$ and $m$,
and where
\[
\beta:=\max\biggl\{ \frac{1}{\sqrt{n}}, \sup_{r, u} \sum
_{v} \bigl\llvert M^{(r)}_{uv}
\bigr\rrvert \biggr\}.
\]
\end{theorem}

\begin{remark}\label{rem34}
We will use Corollary~\ref{cor1} to prove \eq{109} and \eq{111}.
The bounds in \eq{109} and \eq{111} have fewer terms than the bound
in \eq{7} because we will make use of the inequality
$d\leq\alpha\beta$ for $\alpha$ and $\beta$ defined in \eq{5}, which
follows from \eq{106} and the assumption that $\Var(W_r)=1$ for each $r$.
\end{remark}

As a corollary of Theorem~\ref{thm37}, we prove the joint asymptotic
normality of the number of descents and inversions of $\pi$; the rate obtained
is best possible.

\begin{corollary}\label{cor2}
Let $\Des(\pi)$ and $\Inv(\pi)$ be the number of descents and
inversions of
$\pi$, and let
\[
W=(W_1, W_2)^t=\biggl(\frac{\Inv(\pi)-\fracd{1}{2}{n \choose2}}{\sqrt
{\fracb{n(n-1)(2n+5)}{72}}},
\frac{\Des(\pi)-\fracb{n-1}{2}}{\sqrt{\fracb{n+1}{12}}} \biggr)^t.
\]
Then
\[
d_c\bigl(\law( W), \law( Z) \bigr)\leq\frac{C}{\sqrt{n}}
\]
for some absolute constant $C$, where $Z$ is a $2$-dimensional standard Gaussian
vector.
\end{corollary}

\begin{pf}
Set
\[
M^{(1)}_{uv}= \sqrt{\frac{18}{n(n-1)(2n+5)}} \times \cases{ -1&\quad
if $v>u$,
\cr
+1&\quad if $v<u$,
\cr
0 &\quad otherwise, }
\]
and set
\[
M^{(2)}_{uv}= \sqrt{\frac{3}{n+1}} \times \cases{ - 1&\quad
if $v=u+1$,
\cr
+ 1& \quad if $v=u-1$,
\cr
0 &\quad otherwise. }
\]
Hence, we can take $m=1$ in Theorem~\ref{thm37}.
Let $\tau=\pi^{-1}$, which is again a uniform random permutation of
size $n$.
It can be easily verified that $W_1=\sum_{i<j} M^{(1)}_{\tau(i)\tau
(j)}$ and\vspace*{2pt}
$W_2=\sum_{i<j} M^{(2)}_{\tau(i)\tau(j)}$. From Lemma~\ref{lem2},
$\Var(W_1)=\Var(W_2)=1$ and $\llvert \Cov(W_1,W_2)\rrvert \leq C/n$. Moreover,
$\beta
$ as
defined in \eq{18} is smaller than $C/\sqrt{n}$. Therefore, the
corollary is
proved by applying Theorem~\ref{thm37}.
\end{pf}

To prove Theorem~\ref{thm3}, we need the following lemma, the proof of
which is straightforward and therefore omitted.

\begin{lemma}\label{lem3}
For $1\leq r, s, t \leq d$ and $\beta$ defined in \eq{18}, we
have
%
\begin{eqnarray}
\label{19} \sum_{u_1,\dots,u_6}\bigl\llvert
M^{(r)}_{u_1u_2}M^{(s)}_{u_1u_3}M^{(r)}_{u_4u_5}M^{(s)}_{u_4u_6}
\bigr\rrvert &\leq& n^2\beta^4,
\\
\label{20} \mathop{ \sum_{\llvert \{u_1,u_2,u_3\}\rrvert =3, \llvert \{u_4,u_5,u_6\}\rrvert =3,}}\limits
_{\llvert \{u_1,\dots,u_6\}\rrvert  \leq5} \bigl
\llvert M^{(r)}_{u_1
u_2}M^{(s)}_{u_1u_3}M^{(r)}_{u_4u_5}M^{(s)}_{u_4u_6}
\bigr\rrvert &\leq&9 n \beta ^4,
\\
\label{21} \sum_{u_1,\dots,u_8}\bigl\llvert
M^{(r)}_{u_1u_2}M^{(s)}_{u_1u_3}
M^{(t)}_{u_1u_4}M^{(r)}_{u_5u_6}M^{(s)}_{u_5u_7}M^{(t)}_{u_5u_8}
\bigr\rrvert &\leq& n^2\beta^6,
\\
\label{22} \sum_{\llvert \{u_1,\dots,u_8\}\rrvert  \leq
7}\bigl\llvert
M^{(r)}_{u_1u_2}M^{(s)}_{u_1u_3}
M^{(t)}_{u_1u_4}M^{(r)}_{u_5u_6}M^{(s)}_{u_5u_7}M^{(t)}_{u_5u_8}
\bigr\rrvert &\leq& 22 n^2 \beta^4 \beta_2,
\end{eqnarray}
where $\sum_{\llvert \{u_1,\dots,u_k\}\rrvert  \leq k-1}$
stands for summation
over all
tuples $(u_1,\dots,u_k)$ for which at least two components are equal.
\end{lemma}

\begin{pf*}{Proof of Theorem~\ref{thm3}}
We adopt the construction of $W'$ from Fulman \cite{Fulman2004}. Let $I$ be uniformly
chosen from $\{1,\dots,n\}$ and independently of $\pi$. Given $I$, we define
$\pi'$ as
$\pi\circ(I,I+1,\dots, n)$ where $(I,I+1,\dots, n)$ denotes the mapping
$I\mapsto I+1 \mapsto\cdots\mapsto n\mapsto I$, while keeping the rest
identical. As $\pi$ and $\pi'$ both are uniformly distributed, $W$ and
$W'$ have
the same marginal distribution (but are not necessarily exchangeable).
Fulman \cite{Fulman2004} showed that with $\lambda=2/n$
\[
\E^{\pi} \bigl(W'-W\bigr)=-\lambda W.
\]
Following Remark~\ref{rem2}, the bound \eq{7} holds with $D=W'-W$
and $G=\frac{1}{2}\lambda^{-1}D=nD/4$ (cf. Section~\ref{sec5}).
From the construction of $W'$, we have (cf. Lemma~4.2.1 of Fulman \cite{Fulman2004})
\[
D_r=-2\sum_{j: j>I} M_{\pi{(I)}\pi{(j)}}^{(r)}
\]
for $r\in\{1,\dots,d\}$.
By the definition of $\beta$ in \eq{18},
%
\begin{equation}
\label{23} \llvert G\rrvert \le C_d n \beta, \qquad\llvert D\rrvert \le
C_d \beta.
\end{equation}
We first prove that
%
\begin{equation}
\label{24} \Var\E^\pi( D_{r} D_{s})\leq
\frac{C_d\beta^4}{n}.
\end{equation}
From the construction of $W'$,
\begin{eqnarray*}
&&\Var\E^\pi(D_{r} D_{s})
\\
&&\quad= \Var\Biggl( \frac{4}{n}\sum_{i=1}^n
\mathop{\sum_{j_1,j_2: }}\limits
_{ j_1,j_2>i} M^{(r)}_{\pi(i)\pi
(j_1)}M^{(s)}_{\pi(i)\pi(j_2)}
\Biggr)
\\
&&\quad=\frac{16}{n^2}\Var\Biggl( \sum_{i=1}^n
\sum_{j: j>i} M^{(r)}_{\pi(i)
\pi(j)}M^{(s)}_{\pi(i)\pi(j)}+
\sum_{i=1}^n \mathop{\sum
_{j_1,j_2:
}}\limits
_{ j_1,j_2>i, j_1\neq j_2} M^{(r)}_{\pi(i) \pi(j_1)}M^{(s)}_{\pi(i)\pi(j_2)}
\Biggr).
\end{eqnarray*}
Using antisymmetry, it is not difficult to see that the first double
sum in the last line is constant. Hence, we only
need to show that
%
\begin{eqnarray}
\label{25} &&\sum_{\llvert \{i,j_1,j_2\}\rrvert =3, \llvert \{k,l_1,l_2\}\rrvert =3 } \bigl\llvert
K^{(r,s)}_{ij_1j_2kl_1l_2}\bigr\rrvert\nonumber\\
&&\quad:=\sum_{\llvert \{i,j_1,j_2\}\rrvert =3, \llvert \{k,l_1,l_2\}\rrvert =3 } \bigl\llvert \Cov
\bigl(M^{(r)}_{\pi
(i) \pi(j_1)}M^{(s)}_{\pi(i)\pi(j_2)},
M^{(r)}_{\pi(k) \pi(l_1)}M^{(s)}_{\pi(k)\pi(l_2)}\bigr) \bigr
\rrvert \\
&&\quad\leq C_d n\beta^4.\nonumber
\end{eqnarray}
We consider the cases $\llvert \{i,j_1,j_2,k,l_1,l_2\}\rrvert =6$ and $\llvert \{
i,j_1,j_2,k,l_1,l_2\}\rrvert \leq5$ separately.
For the first case, we have
\begin{eqnarray*}
K^{(r,s)}_{ij_1j_2kl_1l_2}&=&\frac{1}{(n)_6}\sum
_{\llvert \{
u,v_1,v_2,w,z_1,z_2\}\rrvert =6} M^{(r,s)}_{uv_1v_2wz_1z_2} - \frac{1}{((n)_3)^2}
\mathop{\sum_{\llvert \{u,v_1,v_2\}\rrvert =3}}\limits
_{ \llvert \{
w,z_1,z_2\}\rrvert =3} M^{(r,s)}_{uv_1v_2wz_1z_2}
\\
&= &\biggl( \frac{1}{(n)_6}- \frac{1}{((n)_3)^2} \biggr) \sum
_{\llvert \{
u,v_1,v_2,w,z_1,z_2\}\rrvert =6} M^{(r,s)}_{uv_1v_2wz_1z_2}
\\
&&{}- \frac{1}{((n)_3)^2} \mathop{\sum_{\llvert \{u,v_1,v_2\}\rrvert =3, \llvert \{
w,z_1,z_2\}\rrvert =3 }}_{ \llvert \{u,v_1,v_2,w,z_1,z_2\}\rrvert \leq5}
M^{(r,s)}_{uv_1v_2wz_1z_2},
\end{eqnarray*}
where $
M^{(r,s)}_{uv_1v_2wz_1z_2}:=M^{(r)}_{uv_1}M^{(s)}_{uv_2}M^{(r)}_{wz_1}M^{(s)}_{wz_2}$.
By \eq{19} and \eq{20},
\[
\bigl\llvert K^{(r,s)}_{ij_1j_2kl_1l_2}\bigr\rrvert \leq C_d
\frac{\beta^4}{n^5}.
\]
Next, we consider the case $\llvert \{i,j_1,j_2,k,l_1,l_2\}\rrvert \leq5$. Let $\t
{\pi}$ be an independent copy
of $\pi$. Again by \eq{19} and \eq{20},
\begin{eqnarray*}
&&\sum_{\llvert \{i,j_1,j_2\}\rrvert =3, \llvert \{k,l_1,l_2\}\rrvert =3 \atop\llvert \{
i,j_1,j_2,k,l_1,l_2\}\rrvert \leq5 } \bigl\llvert K^{(r,s)}_{ij_1j_2kl_1l_2}
\bigr\rrvert
\\
&&\quad\leq\E \biggl[ \sum_{\llvert \{i,j_1,j_2\}\rrvert =3, \llvert \{k,l_1,l_2\}\rrvert =3 \atop\llvert \{
i,j_1,j_2,k,l_1,l_2\}\rrvert \leq5 } \bigl( \bigl\llvert
M^{(r,s)}_{\pi(i)\pi(j_1)\pi
(j_2)\pi(k)\pi(l_1)\pi(l_2)}\bigr\rrvert + \bigl\llvert M^{(r,s)}_{\pi(i)\pi(j_1)\pi
(j_2)\t{\pi
}(k)\t{\pi}(l_1)\t{\pi}(l_2)}
\bigr\rrvert \bigr) \biggr]
\\
&&\quad=\E \biggl[ \sum_{\llvert \{u,v_1,v_2\}\rrvert =3, \llvert \{w,z_1,z_2\}\rrvert =3} \bigl\llvert
M^{(r,s)}_{uv_1v_2wz_1z_2}\bigr\rrvert \bigl(\I \bigl(\bigl\llvert
\{u,v_1,v_2,w,z_1,z_2\}\bigr
\rrvert \leq5 \bigr)
\\
&&\hspace*{160pt}\qquad{}+\I \bigl(\bigl\llvert \pi^{-1}\bigl(\{u,v_1,v_2
\}\bigr) \cap\t{\pi}^{-1}\bigl(\{ w,z_1,z_2\}
\bigr) \bigr\rrvert \leq5 \bigr) \bigr) \biggr]
\\
&&\quad\leq C_d n \beta^4.
\end{eqnarray*}
Therefore, we have
proved
\eq{25}, and thus \eq{24}.
Again from the construction of $W'$, we can write
\[
\Var\E^\pi(D_r D_s D_t)=
\frac{64}{n^2}\sum_{j_1>i, j_2>i, j_3>i
\atop
l_1>k, l_2>k, l_3>k} K^{(r,s,t)}_{i j_1 j_2 j_3 k l_1 l_2 l_3},
\]
where
\begin{eqnarray*}
K^{(r,s,t)}_{i j_1 j_2 j_3 k l_1 l_2 l_3}
:=\Cov\bigl(M^{(r)}_{\pi(i) \pi(j_1)}M^{(s)}_{\pi(i)\pi
(j_2)}M^{(t)}_{\pi(i)\pi(j_3)},
M^{(r)}_{\pi(k) \pi(l_1)}M^{(s)}_{\pi(k)\pi(l_2)}M^{t}_{\pi(k)\pi(l_3)}
\bigr).
\end{eqnarray*}
By the same argument as for $K^{(r,s)}_{ij_1j_2kl_1l_2}$, and using
the
bounds \eq{21} and \eq{22} instead of \eq{19} and \eq{20}, we can prove
%
\begin{equation}
\label{26} \Var\E^\pi(D_{r} D_{s}
D_{t})\leq C_d \beta^4 \beta_2.
\end{equation}
Applying the bounds \eq{23}, \eq{24} and \eq{26} in \eq{7} and using
$1 \leq C_d n \beta^2$ by Remark~\ref{rem34} prove the theorem.
\end{pf*}

Next, we prove Theorem~\ref{thm37}.

\begin{pf*}{Proof of Theorem~\ref{thm37}}
Let $\tau=\pi^{-1}$.
From Diaconis \cite{Diaconis1988}, $W_1$ can be expressed as
\[
W_1=\sum_{u=1}^n
\frac{1}{\sqrt{\fracb{n(n-1)(2n+5)}{72}}} \biggl( \xi _u-\frac{n-u}{2} \biggr) =:\sum
_{u=1}^n X^{(1)}_u,
\]
where $\xi_1$ is the minimum number of pairwise adjacent
transpositions taking
$\tau(1)$ to the first position, $\xi_2$ is the minimum number of pairwise
adjacent transpositions taking $\tau(2)$ to the second position after the
first step is done, etc. Because $\tau$ is a uniform permutation, $\{
\xi
_1,\dots,
\xi_n\}$ are independent random variables with $\xi_u\sim\Uniform
\{0,\dots, n-u\}$ for $1\le u\le n$.
For $2\leq r\leq d$, by the assumption that $M^{(r)}_{uv}=0$ if $\llvert u-v\rrvert >m$,
\begin{eqnarray*}
W_r&=&\sum_{i<j} M^{(r)}_{\pi(i)\pi(j)}=
\sum_{u,v: \llvert u-v\rrvert \leq m\atop
\pi
^{-1}(u)<\pi^{-1}(v)} M^{(r)}_{uv}
\\
&=&\sum_{u=1}^n \biggl( \sum
_{v: \llvert u-v\rrvert \leq m}M^{(r)}_{uv} \I\bigl[\tau (u)<
\tau(v)\bigr] \biggr) =:\sum_{u=1}^n
X^{(r)}_u.
\end{eqnarray*}
Let $X_u=(X^{(1)}_u,\dots,X^{(d)}_u)^t$. Then $W=\sum_{u=1}^n X_u$.
In the above pairwise transposition process, if we know $\{\xi_v\dvtx
1\leq
v\leq u+m \}$, then we can reconstruct the positions of $\{\tau(v)\dvtx
\llvert v-u\rrvert \leq m \}$. Observe that the relative order of $\{\tau(v)\dvtx
\llvert v-u\rrvert \leq m \}$ does not depend on $\{\xi_v\dvtx  1\leq v< u-m \}$.
Therefore, $X_u$ is measurable with respect to $\{\xi_v\dvtx  \llvert v-u\rrvert \leq m \}
$ and $W$ can be viewed as a sum of locally dependent random vectors
(cf. Section~\ref{sec8}) with neighbourhood $A_u=\{u-2m, u+2m\}$ for
each $1\leq u\leq n$. For the Stein coupling \eq{112}, we have\vspace*{-1pt}
\[
\llvert G\rrvert \leq C_{d,m}n\beta,\qquad\llvert D\rrvert \leq
C_{d,m}\beta.
\]
Moreover, by the local dependence structure,\vspace*{-1pt}
\begin{eqnarray*}
\Var\E^W(G_r D_s)&\leq&\Var\Biggl( \sum
_{u=1}^n \sum
_{v\in A_u} X^{(r)}_u X^{(s)}_v
\Biggr)
\\
&=& \sum_{u=1}^n \sum
_{w: \llvert w-u\rrvert \leq6m} \Cov\biggl( \sum_{v\in A_u}
X^{(r)}_u X^{(s)}_v, \sum
_{z\in A_w} X^{(r)}_w X^{(s)}_z
\biggr)
\\
&\leq& C_{d,m}n\beta^4,
\end{eqnarray*}
where we used the inequality $\Cov(X,Y)\leq(\E X^2+\E Y^2)/2$.
Similarly,\vspace*{-1pt}
\[
\Var\E^W(G_r D_s D_t)\leq
C_{d,m} n \beta^6.
\]
The bound \eq{111} is proved by applying the above bounds to \eq{7} and
using $1\leq C_{d,m} n \beta^2$ by Remark~\ref{rem34}.\vspace*{-1pt}
\end{pf*}

\section{Proof of main theorem} \label{sec4}

For given test function $h$, we consider the Stein equation\vspace*{-1pt}
%
\begin{equation}
\label{27} \D f(w)-w^t\nabla f(w)=h(w)-\E h(Z),\qquad w\in
\IR^d,
\end{equation}
where $\D$ denotes the Laplacian operator and $\nabla$ the gradient operator.
If $h$ is not continuous (like the indicator function of a convex set),
then $f$
is not smooth enough to apply Taylor expansion to the necessary degree, so
more refined techniques are necessary.

We follow the smoothing technique of Bentkus \cite{Bentkus2003}. Recall that
$\mathcal{A}$ is the collection of all convex sets in $\IR^d$. For
$A\in\mathcal{A}$, let $h_A(x)=I_A(x)$, and define the
smoothed function\vspace*{-1pt}
%
\begin{equation}
\label{28} h_{A,\eps}(w) = \psi\biggl(\frac{\dist(w, A)}{\eps}\biggr),
\end{equation}
where $\dist(w,A) = \inf_{v\in A}\llvert  w-v\rrvert  $ and\vspace*{-1pt}
%
\begin{equation}
\psi(x)= \cases{ 1, &\quad $x<0$,
\cr
1-2x^2, &\quad $0\leq x<
\frac{1}{2}$,
\cr
2(1-x)^2, &\quad $\frac{1}{2} \leq x <1$,\vspace*{2pt}
\cr
0, &\quad $1\leq x$. }
\end{equation}
Define also\vspace*{-1pt}
\[
A^\eps= \bigl\{x\in\IR^d \dvtx  \dist(x,A)\leq\eps\bigr\},\qquad
A^{-\eps} = \bigl\{x\in A \dvtx  \dist\bigl(x,\IR^d\setminus A
\bigr) > \eps\bigr\}
\]
(note that in general $(A^{-\eps})^\eps\neq A$).

\begin{lemma}[(Lemma~2.3 of Bentkus \cite{Bentkus2003})] The
function $h_{A,\eps}$ as defined above
has the following properties:
%
\begin{eqnarray}\label{29}
\mbox{\emph{(i)}}&& h_{A,\eps}(w)=1\qquad \mbox{for all }w\in A,
\\
\label{30}
\mbox{\emph{(ii)}}&& h_{A,\eps}(w)=0\qquad \mbox{for all }w\in\IR^d\setminus
A^\eps,
\\
\label{31}
\mbox{\emph{(iii)}}&& 0\leq h_{A,\eps}(w) \leq1\qquad \mbox{for all }w\in A^\eps
\setminus A,
\\
\label{32}
\mbox{\emph{(iv)}}&& \bigl\llvert \nabla h_{A,\eps} (w)\bigr\rrvert \leq2
\eps^{-1}\qquad \mbox{for all }w\in\IR^d,
\\
\label{33}
\mbox{\emph{(v)}} && \bigl\llvert \nabla h_{A,\eps} (v)-\nabla h_{A,\eps}
(w)\bigr\rrvert \leq8\llvert v-w\rrvert \eps^{-2}\qquad \mbox{for all }v,w\in
\mrr.
\end{eqnarray}
\end{lemma}

\begin{lemma}\label{lem4}
For any $d$-dimensional random vector $W$,
%
\begin{equation}
\label{34} d_c\bigl(\law(W), \law(Z)\bigr) \leq4d^{1/4}
\eps+ \sup_{A\in\mathcal{A}}\bigl\llvert \E h_{A,\eps}(W)-\E
h_{A,\eps
}(Z)\bigr\rrvert.
\end{equation}
\end{lemma}

\begin{pf} By (2.2) of Bentkus \cite{Bentkus2003}, for any $\eps>0$,
\begin{eqnarray*}
d_c\bigl(\law(W), \law(Z)\bigr) &\leq&\sup_{A\in\mathcal{A}}
\bigl\llvert \E h_{A,\eps}(W)-\E h_{A,\eps}(Z)\bigr\rrvert
\\
&&{}+\sup_{A\in\mathcal{A}} \max\bigl\{\P\bigl(Z\in A^\eps\setminus
A\bigr),\P \bigl(Z\in A\setminus A^{-\eps}\bigr)\bigr\}.
\end{eqnarray*}
From Ball \cite{Ball1993} and Bentkus \cite{Bentkus2003}, we have
%
\begin{equation}
\label{35} \sup_{A\in\mathcal{A}}\max\bigl\{\P\bigl(Z\in
A^\eps\setminus A\bigr),\P\bigl(Z\in A\setminus A^{-\eps}\bigr)
\bigr\} \leq4d^{1/4} \eps
\end{equation}
(the dependence on $d$ in \eq{34} is optimal; see Bentkus \cite{Bentkus2003}).
\end{pf}

Fix now $\eps$ and a convex $A\subset\IR^d$. It can be verified directly
that with
\[
g_{A,\eps}(w,s)=-\frac{1}{2(1-s)} \int_{\IR^d}
\bigl[h_{A,\eps}(\sqrt{1-s}w+\sqrt{s}z)-\E h_{A,\eps} (Z)\bigr]
\phi(z)\,\mathrm{d}z,
\]
a solution to \eq{27} is (cf. G{\"o}tze \cite{Goetze1991})
%
\begin{equation}
\label{103} f_{A,\eps}(w)=\int_0^1
g_{A, \eps}(w,s)\,\mathrm{d}s,
\end{equation}
where $\phi$ is the density function of the
$d$-dimensional standard normal
distribution. In what follows, we keep the dependence on $A$ and $\eps$
implicit and write $g=g_{A, \eps}, f=f_{A,\eps}$ and $h=h_{A,\eps}$.
For real-valued functions on $\IR^d$ we
will write $f_i(x)$ for $\partial f(x)/\partial x_i$, $f_{ij}(x)$ for
$\partial^2 f(x)/(\partial x_i\,\partial x_j)$ and so forth. Also we
write $g_i(w,s)=\partial g(w,s)/\partial w_i$ and so on.

Using this notation and the integration by parts formula, we have for
$1\leq i,j, k\leq d$ that
%
\begin{eqnarray}
\label{36} g_{ij} (w,s) & =& - \frac{1}{2s} \int
_{\IR^d} h (\sqrt{1-s}w+\sqrt{s}z) \phi_{ij} (z)\,\mathrm{d}z\nonumber
\\[-8pt]\\[-8pt]
&=& \frac{1}{2\sqrt{s}} \int_{\IR^d} h_j(
\sqrt{1-s}w+\sqrt{s}z) \phi_i (z)\,\mathrm{d}z\nonumber
\end{eqnarray}
and
%
\begin{eqnarray}
\label{102} g_{ijk} (w,s) & =& \frac{\sqrt{1-s}}{2s^{3/2}} \int
_{\IR^d} h (\sqrt{1-s}w+\sqrt{s}z) \phi_{ijk} (z)\,\mathrm{d}z\nonumber
\\[-8pt]\\[-8pt]
&=& \frac{\sqrt{1-s}}{2\sqrt{s}} \int_{\IR^d} h_{jk}(
\sqrt{1-s}w+\sqrt{s}z) \phi_i (z)\,\mathrm{d}z.\nonumber
\end{eqnarray}

\begin{lemma}\label{lem5}
For each map $a\dvtx  \{1,\dots,d\}^k\rightarrow\IR$, we have
%
\begin{equation}
\label{104} \int_{\mrr} \Biggl( \sum
_{i_1,\dots, i_k=1}^d a(i_1,\dots, i_k)
\frac
{\phi
_{i_1\dots i_k}(z)}{\phi(z)} \Biggr)^2\phi(z)\,\mathrm{d}z \leq k! \sum
_{i_1,\dots, i_k=1}^d \bigl( a(i_1,\dots,
i_k) \bigr)^2.
\end{equation}
\end{lemma}

\begin{pf}
We will prove that
%
\begin{equation}
\label{101} \int_{\mrr} \frac{\phi_{i_1\cdots i_k}(z)}{\phi(z)}
\frac{\phi
_{j_1\cdots
j_k}(z)}{\phi(z)}\phi(z)\,\mathrm{d}z =\sum_{\pi}
\delta_{i_{\pi(1)} j_1}\cdots\delta_{i_{\pi(k)} j_k},
\end{equation}
where the summation is over all permutations of the set $\{1,\dots, k\}
$ and $\delta$ is the Kronecker delta. By \eq{101},
\begin{eqnarray*}
 \int_{\mrr} \Biggl( \sum_{i_1,\dots, i_k=1}^d
a(i_1,\dots, i_k)\frac{\phi
_{i_1\dots i_k}(z)}{\phi(z)}
\Biggr)^2\phi(z)\,\mathrm{d}z
&=&\sum_\pi\sum_{i_1,\dots, i_k=1}^d
a(i_1,\dots, i_k) a(i_{\pi
(1)},\dots,
i_{\pi(k)})
\\
&\leq& k! \sum_{i_1,\dots, i_k=1}^d \bigl(
a(i_1,\dots, i_k) \bigr)^2.
\end{eqnarray*}
To prove \eq{101}, we observe that
\begin{eqnarray*}
&& \int_{\mrr} \frac{\phi_{i_1\cdots i_k}(z)}{\phi(z)} \frac{\phi
_{j_1\cdots j_k}(z)}{\phi(z)}\phi(z)
\,\mathrm{d}z
\\
&&\quad=\frac{\partial^{2k}}{\partial x_{i_1}\cdots\partial x_{i_k}
\partial
y_{j_1}\cdots\partial y_{j_k}} \Big|_{x=y=0}\int_{\mrr}
\frac{\phi(z+x)}{\phi(z)}\frac{\phi
(z+y)}{\phi
(z)}\phi(z)\,\mathrm{d}z
\\
&&\quad=\frac{\partial^{2k}}{\partial w_1\cdots\partial w_{2k}} \Big|_{x=y=0} \mathrm{e}^{\langle x,y\rangle}
\end{eqnarray*}
where $w_s=x_{i_s}$ and $w_{s+k}=y_{j_s}$ for $s=1,2,\dots, k$.
By Fa\`a di Bruno's formula (see Hardy \cite{Hardy2006}), the latter
expression equals
\[
\sum_{P_1,\dots, P_m}\frac{\partial^{\llvert P_1\rrvert }\langle x,y\rangle}{\Pi_{s\in P_1}
\partial w_s} \dots
\frac{\partial^{\llvert P_m\rrvert }\langle x,y\rangle}{\Pi_{s\in P_m} \partial w_s} \Big|_{x=y=0},
\]
where the summation is over all unordered partitions of the set $\{
1,2,\dots, 2k\}$ and $\llvert \cdot\rrvert $ denotes the cardinality. However, the
summand is non-zero if and only if each $P_r$ is of the form $\{s,t+k\}
$ where $i_s=j_t$. This proves \eq{101}.
\end{pf}

\begin{pf*}{Proof of Theorem~\ref{thm1}}
Fix $A\in\mathcal{A}$ and $\eps>0$ (to be chosen
later) and let $f=f_{A,\eps}$ be the solution to the Stein
equation \eq{27}
with respect to $h=h_{A,\eps}$ as defined by \eq{28}. Let
\[
\kappa:= d_c\bigl(\law(W),\law(Z)\bigr).
\]
Adding and subtracting the corresponding terms, we have for
$g(w,s)=g_{A,\eps}(w,s)$ in \eq{103},
\begin{eqnarray*}
&&\E\bigl\{\Delta g(W,s)- W^t\nabla g(W,s)\bigr\}
\\
&&\quad= \E\bigl\{G^t \nabla g\bigl(W',s
\bigr)-G^t \nabla g(W,s)-W^t \nabla g(W,s)\bigr\}\\
&&\qquad{}+\sum_{i,j=1}^d \E\bigl\{(
\delta_{ij}- G_{i}D_{j})g_{ij}(W,s)
\bigr\}
\\
&&\qquad{}-\E\Biggl\{\sum_{i=1}^d
G_{i} g_i\bigl(W',s\bigr)-\sum
_{i=1}^d G_{i}g_i(W,s)-
\sum_{i,j=1}^d G_{i}D_{j}
g_{ij}(W,s)\Biggr\}
\\
&&\quad=: R_0(s) + R_1(s) - R_2(s).
\end{eqnarray*}
As $(W,W',G)$ is a Stein coupling, clearly $R_0(s)\equiv0$. Therefore,
by \eq{27},
\[
\E h(W)-\E h(Z)=\int_0^1
\bigl(R_1(s)-R_2(s)\bigr)\,\mathrm{d}s.
\]
To estimate $\int_0^1 R_1(s)\,\mathrm{d}s$, we consider the cases $\eps^2<
s\leq
1$ and $0< s\leq\eps^2$ separately.
For the first case, we use the first expression of $g_{ij}(w,s)$ in \eq
{36} and obtain
\begin{eqnarray*}
\int_{\eps^2}^1 R_1(s)\,\mathrm{d}s & =&\sum
_{i,j=1}^d \E\int_{\eps^2}
^1 \biggl(-\frac{1}{2s}\biggr) \int_{\mrr}
\bigl[\E ^W (\delta_{ij} -G_i
D_{j})\bigr]
\\
&&{}\times\bigl[h (\sqrt{1-s}W+\sqrt{s}z)-h (\sqrt{1-s}W)\bigr]
\phi_{ij}(z)\,\mathrm{d}z\,\mathrm{d}s,
\end{eqnarray*}
where we used the fact
that $\int_{\mrr} \phi_{ij}(z)\,\mathrm{d}z=0$.
By the Cauchy--Schwarz inequality, \eq{104} and \eq{106},
\begin{eqnarray*}
&&\sum_{i,j=1}^d \E\int_{\mrr}
\bigl[\E ^W (\delta_{ij} -G_i
D_{j})\bigr] \bigl[h (\sqrt{1-s}W+\sqrt{s}z)-h (\sqrt{1-s}W)\bigr]
\phi_{ij}(z)\,\mathrm{d}z
\\
&&\quad\leq \Biggl\{ \E\int_{\mrr} \Biggl( \sum
_{i,j=1}^d \bigl[ \E ^W (
\delta_{ij} -G_i D_{j})\bigr]
\frac{\phi_{ij}(z)}{\phi(z)} \Biggr)^2 \phi(z)\,\mathrm{d}z \Biggr\}^{1/2}
\\
&&\qquad{}\times \biggl\{ \int_{\mrr} \E\bigl[h (\sqrt {1-s}W+
\sqrt{s}z)-h (\sqrt{1-s}W)\bigr]^2 \phi(z)\,\mathrm{d}z \biggr\}^{1/2}
\\
&&\quad\leq\sqrt{2} B_2 \biggl\{ \int_{\mrr} \E\bigl[h
(\sqrt {1-s}W+\sqrt{s}z)-h (\sqrt{1-s}W)\bigr]^2 \phi(z)\,\mathrm{d}z \biggr
\}^{1/2}.
\end{eqnarray*}
From the definition of $\kappa$ and the concentration inequality of
the standard
$d$-dimensional Gaussian distribution (cf. \eq{35}), we have
%
\begin{eqnarray}
\label{105} &&\E\bigl\{h (\sqrt{1-s}W+\sqrt{s}z)-h (\sqrt{1-s} W)\bigr
\}^2\nonumber
\\
&&\quad\leq\E\bigl\{ \I\bigl[\dist\bigl(\sqrt{1-s}W, A^\eps \setminus A
\bigr)\leq \sqrt{s} \llvert z\rrvert \bigr]\bigr\}\nonumber\\[-8pt]\\[-8pt]
&&\quad\leq\P\bigl\{ \I\bigl[\dist\bigl(\sqrt{1-s}Z, A^\eps \setminus A
\bigr)\leq \sqrt{s} \llvert z\rrvert \bigr]\bigr\} +2d_c\bigl(
\law(W), \law(Z)\bigr)\nonumber
\\
&&\quad \leq4d^{1/4} \biggl(\frac{\eps}{\sqrt{1-s}} + 2\sqrt{\frac{s}{1-s}}
\llvert z\rrvert \biggr) +2\kappa.\nonumber
\end{eqnarray}
Using the Cauchy--Schwarz inequality, the bound \eq{32},
the simple inequality\break $\sqrt{a_1+a_2+a_3}\leq\sqrt{a_1}+\sqrt
{a_2}+\sqrt{a_3}$
for $a_1,a_2,a_3\ge
0$, and $\int_{\mrr} \llvert z\rrvert ^{1/2}\phi(z)\,\mathrm{d}z\leq d^{1/4}$, we have
\begin{eqnarray*}
\biggl\llvert \int_{\eps^2}^1 R_1(s)
\,\mathrm{d}s \biggr\rrvert
&\leq& C B_2 \int_{\eps^2} ^1
\frac{1}{s} \int_{\mrr} \biggl(d^{1/4}
\frac{\eps}{\sqrt{1-s}}+d^{1/4}\sqrt{\frac{s}{1-s}}\llvert z\rrvert +
\kappa\biggr)^{1/2}\phi(z)\,\mathrm{d}z\,\mathrm{d}s
\\
&\leq& C B_2 \bigl( d^{1/8} \eps^{1/2} \llvert \log
\eps\rrvert +d^{3/8} + \kappa^{1/2} \llvert \log\eps\rrvert
\bigr),
\end{eqnarray*}
where we used
$\int_{\epsilon^2}^1 \frac{1}{s(1-s)^{1/4}}\,\mathrm{d}s \leq C\llvert \log\epsilon\rrvert $
and $\int_{\epsilon^2}^1 \frac{1}{s^{3/4} (1-s)^{1/4}}\,\mathrm{d}s \leq C$.

For the case $0< s\leq\eps^2$, we use the second expression of
$g_{ij}(w,s)$ in \eq{36}, the Cauchy--Schwarz inequality, the bound \eq
{32} and \eq{104}, and obtain
\begin{eqnarray*}
&&\biggl\llvert \int_{0}^{\eps^2} R_1(s)
\,\mathrm{d}s \biggr\rrvert
\\
&&\quad=\Biggl\llvert \sum_{i,j=1}^d \E\int
_0^{\eps^2} \frac{1}{2\sqrt{s}} \int
_{\mrr} \bigl[\E^W (\delta_{ij}-G_i
D_{j})\bigr] h_j (\sqrt{1-s} W+\sqrt{s}z)
\phi_i(z)\,\mathrm{d}z\,\mathrm{d}s \Biggr\rrvert
\\
&&\quad\leq\eps\Biggl\llvert \E\int_{\mrr}\frac{2}{\eps}\sum
_{i=1}^d \Biggl\{ \sum
_{j=1}^d \bigl[\E^W(
\delta_{ij}-G_i D_j)\bigr]^2
\Biggr\}^{1/2}\frac{\phi
_i(z)}{\phi
(z)}\phi(z)\,\mathrm{d}z \Biggr\rrvert
\\
&&\quad\leq2\E \Biggl\{ \int_{\mrr} \Biggl\{\sum
_{i=1}^d \Biggl\{ \sum
_{j=1}^d \bigl[\E ^W(
\delta_{ij}-G_i D_j)\bigr]^2
\Biggr\}^{1/2} \frac{\phi_i(z)}{\phi(z)} \Biggr\} ^2 \phi(z)\,\mathrm{d}z
\Biggr\}^{1/2}
\\
&&\quad\leq2B_2,
\end{eqnarray*}
where the factor $\epsilon$ in the first inequality comes from $\int_0^{\epsilon^2}\frac{1}{2\sqrt{s}}\,\mathrm{d}s\leq\epsilon$.
Therefore,
\[
\biggl\llvert \int_0^1 R_1(s)\,\mathrm{d}s
\biggr\rrvert \leq C B_2 \bigl( d^{1/8} \eps^{1/2}
\llvert \log\eps\rrvert +d^{3/8} + \kappa^{1/2} \llvert \log\eps
\rrvert \bigr).
\]

In order to estimate $\int_0^1 R_2(s)\,\mathrm{d}s$, let $U$ and $V$ be
independent random
variables distributed uniformly on $[0,1]$.
Then
\[
R_2(s)=\E\sum_{i,j,k=1}^d U
G_i D_j D_k g_{ijk}(W+UVD,s).
\]
We again consider the cases $\eps^2< s\leq1$ and $0< s\leq\eps^2$
separately.

For the first case, we use the first expression of $g_{ijk}(w,s)$
in \eq
{102} and obtain
\begin{eqnarray*}
\int_{\eps^2}^1 R_2(s)\,\mathrm{d}s &=&\sum
_{i,j,k=1}^d \E\int_{\eps^2}^1
\frac{\sqrt{1-s}}{2s^{3/2}} \int_{\mrr} \bigl[h(\sqrt{1-s}W +\sqrt{s}z+
\sqrt{1-s} UVD)
\\
&&\hspace*{102pt}{}-h(\sqrt{1-s}W+\sqrt{s}z)\bigr] UG_i D_{j}
D_{k} \phi_{ijk} (z)\,\mathrm{d}z\,\mathrm{d}s
\\
&&{}+\sum_{i,j,k=1}^d \E\int
_{\eps^2}^1 \frac{\sqrt{1-s}}{2s^{3/2}} \int
_{\mrr} h(\sqrt{1-s}W+\sqrt{s}z)
\\
&&\hspace*{8pt}{}\times U\bigl[\E^W(G_{i}D_{j}D_{k})-
\E(G_{i}D_{j}D_{k})\bigr]\phi_{ijk}
(z)\,\mathrm{d}z\,\mathrm{d}s
\\
&&{}+\sum_{i,j,k=1}^d \E\int
_{\eps^2}^1 \frac{\sqrt{1-s}}{2s^{3/2}} \int
_{\mrr} \bigl[h(\sqrt{1-s}W+\sqrt{s}z)
\\
&&\hspace*{112pt}{} -h(\sqrt{1-s}Z+\sqrt{s}z)\bigr] U\E(G_{i}D_{j}D_{k})
\phi_{ijk}(z)\,\mathrm{d}z\,\mathrm{d}s
\\
&&{}+\sum_{i,j,k=1}^d \E\int
_{\eps}^1 U \E(G_i D_j
D_k) g_{ijk}(Z,s)\,\mathrm{d}z\,\mathrm{d}s
\\
&=:& R_{2,1,1}+R_{2,1,2}+R_{2,1,3}+R_{2,1,4},
\end{eqnarray*}
where $Z$ is an independent $d$-dimensional standard Gaussian random vector.
Now, it is straightforward to verify that for any $u,v,w,z\in\IR^d$
%
\begin{eqnarray}
\label{37}  \sum_{i,j,k=1}^d
u_i v_j w_k \phi_{ijk}(z)
= -u^t z v^tz w^tz \phi(z) +
\bigl(u^tv w^tz + u^tw v^tz +
v^tw u^tz\bigr)\phi(z).
\end{eqnarray}
In bounding $\int_{\eps^2}^1 R_2(s)\,\mathrm{d}s$, the integration with respect
to $s$ is bounded by
$\int_{\epsilon^2}^1 \frac{1}{s^{3/2}}\,\mathrm{d}s\leq C\epsilon^{-1}$.
From \eq{37} and the boundedness condition \eq{5},
\begin{eqnarray*}
\llvert R_{2,1,1}\rrvert &\leq&\E\int_{\eps^2}^1
\frac{\sqrt{1-s}}{2s^{3/2}} \int_{\mrr} \I\bigl[\dist\bigl(\sqrt{1-s}W+
\sqrt{s}z,A^\eps\setminus A\bigr)\leq \sqrt{1-s} \beta\bigr]
\\
&&{} \times U\Biggl\llvert \sum_{i,j,k=1}^d
G_i D_{j} D_{k}\phi _{ijk}(z)
\Biggr\rrvert \,\mathrm{d}z\,\mathrm{d}s
\\
&\leq&\alpha\int_{\eps^2}^1 \frac{\sqrt{1-s}}{4s^{3/2}}\int
_{\mrr} \E \bigl\{ \I \bigl[\dist\bigl(\sqrt{1-s}W+\sqrt{s}z,
A^\eps\setminus A\bigr)\leq \sqrt {1-s}\beta\bigr]
\\
&&\hspace*{82pt}{} \times \bigl[ \E\llvert D\rrvert ^2+ \bigl(
\E^W \llvert D\rrvert ^2-\E\llvert D\rrvert ^2
\bigr) \bigr] \bigr\} \bigl(3\llvert z\rrvert +\llvert z\rrvert ^3
\bigr)\phi(z)\,\mathrm{d}z\,\mathrm{d}s
\\
&\leq& C d^{3/2}\alpha\E\llvert D\rrvert ^2
\eps^{-1} \bigl(\kappa+d^{1/4} (\beta+\eps)
\bigr)+Cd^{3/2}\eps^{-1} \alpha B_1,
\end{eqnarray*}
where in the last inequality we used a similar recursive inequality
as \eq{105} as well as
$\int_{\mrr}\llvert z\rrvert ^3 \phi(z)\,\mathrm{d}z\leq d^{3/2}$.

From the Cauchy--Schwarz inequality and \eq{104},
\begin{eqnarray*}
\llvert R_{2,1,2}\rrvert \leq C\eps^{-1} B_3.
\end{eqnarray*}
From \eq{37} and a recursive inequality as \eq{105},
\begin{eqnarray*}
\llvert R_{2,1,3}\rrvert \leq C\bigl(\kappa\eps^{-1}+d^{1/4}
\bigr) \E\bigl(\llvert G\rrvert \llvert D\rrvert ^2\bigr).
\end{eqnarray*}
For $R_{2,1,4}$, observe that
\begin{eqnarray*}
\E g(Z+w,s)&=&-\frac{1}{2(1-s)}\int_{\mrr} \bigl[ \E h\bigl(
\sqrt {1-s}(Z+w)+\sqrt{s}z\bigr) -\E h(Z) \bigr]\phi(z)\,\mathrm{d}z
\\
&=&-\frac{1}{2(1-s)}\int_{\mrr} h(\sqrt{1-s}w+z)\phi(z)\,\mathrm{d}z+
\frac
{1}{2(1-s)}\E h(Z)
\\
&=&-\frac{1}{2(1-s)}\int_{\mrr} h(x)\phi(x-\sqrt{1-s}w)\,\mathrm{d}x+
\frac
{1}{2(1-s)}\E h(Z).
\end{eqnarray*}
Differentiating and evaluating at $w=0$, we obtain
\[
\E g_{ijk}(Z,s)=\frac{\sqrt{1-s}}{2}\int_{\mrr} h(x)
\phi_{ijk}(x)\,\mathrm{d}x.
\]
Now with
\eq{37},
\[
\llvert R_{2,1,4}\rrvert \leq C \E\bigl(\llvert G\rrvert \llvert D
\rrvert ^2\bigr).
\]

For the case $0<s\leq\eps^2$, we use the second expression of
$g_{ijk}$ in \eq{102}.
From \eq{33} and $\llvert \sum_{i=1}^d G_i \phi_i(z)\rrvert \leq\alpha\llvert z\rrvert \phi(z)$,
\begin{eqnarray*}
&&\biggl\llvert \int_0^{\eps^2} R_2(s)
\,\mathrm{d}s \biggr\rrvert
\\
&&\quad=\Biggl\llvert \sum_{i,j,k=1}^d \E\int
_0^{\eps^2} \frac{\sqrt{1-s}}{2s^{1/2}} \int
_{\mrr} h_{jk} (\sqrt{1-s}W +\sqrt{s}z+\sqrt{1-s} UVD
)
UG_i D_{j} D_{k}
\phi_i(z)\,\mathrm{d}z\,\mathrm{d}s \Biggr\rrvert
\\
&&\quad\leq\frac{8\alpha}{\eps^2} \E\int_0^{\eps^2}
\frac{\sqrt{1-s}}{2s^{1/2}} \int_{\mrr} \I\bigl[\dist\bigl(\sqrt{1-s}W +
\sqrt{s}z, A^\eps\setminus A\bigr)\leq\sqrt{1-s} \beta\bigr]
\\
&&\qquad {}\times \bigl[ \llvert D\rrvert ^2+ \bigl(\E^W
\llvert D\rrvert ^2-\E\llvert D\rrvert ^2 \bigr) \bigr]
\llvert z\rrvert \phi (z)\,\mathrm{d}z\,\mathrm{d}s
\\
&&\quad\leq Cd^{1/2}\alpha\bigl( \E\llvert D\rrvert ^2 \bigl(
\kappa\eps^{-1}+d^{1/4}\beta\eps^{-1}+d^{1/4}
\bigr)+ \eps^{-1} \sqrt{\Var\E^W \llvert D\rrvert
^2}\bigr),
\end{eqnarray*}
where in the last inequality we used a similar recursive inequality
as \eq{105} as well as
$\int_{\mrr}\llvert z\rrvert  \phi(z)\,\mathrm{d}z\leq d^{1/2}$ and $\int_{0}^{\epsilon
^2}\frac
{1}{2s^{1/2}}\,\mathrm{d}s\leq\epsilon$.

Therefore,
\begin{eqnarray*}
\biggl\llvert \int_0^1 R_2(s)\,\mathrm{d}s
\biggr\rrvert \leq C \bigl(d^{3/2}\alpha\E\llvert D\rrvert
^2 \eps^{-1} \bigl(\kappa +d^{1/4} (\beta+\eps)
\bigr)+ d^{3/2} \eps^{-1} \alpha B_1 +
\eps^{-1} B_3\bigr).
\end{eqnarray*}
Collecting the bounds and using the smoothing inequality \eq{34}, we
obtain the following recursive inequality
%
\begin{eqnarray}
\label{38} \kappa &\leq& C \bigl\{ d^{3/2}\alpha\E\llvert D\rrvert
^2 \eps^{-1} \bigl(\kappa +d^{1/4} (\beta+\eps)
\bigr)+d^{3/2} \eps^{-1} \alpha B_1\nonumber
\\[-8pt]\\[-8pt]
&&\hspace*{10pt}{}+ \eps^{-1} B_3 + B_2 \bigl(d^{3/8}
+d^{1/8} \eps^{1/2} \llvert \log\eps\rrvert + \sqrt{\kappa}
\llvert \log\eps\rrvert \bigr)\bigr\} +4d^{1/4} \eps.\nonumber
\end{eqnarray}
Let
\[
\eps= 2C d^{3/2} \alpha\E\llvert D\rrvert ^2+
\beta+d^{5/8} \alpha^{1/2} B_1^{1/2}
+d^{1/8} B_2 +d^{-1/8}B_3^{1/2}
\]
with the same constant $C$ as in \eq{38}. The theorem is proved by solving
the recursive inequality for $\kappa$ and observing that
as long as $\eps$ is smaller than an absolute constant, $\eps^{1/2}
\llvert \log\eps\rrvert  \leq C$ and $\kappa^{1/2} \llvert \log\eps\rrvert  \leq C d^{1/8}$,
the latter follows by solving the recursive inequality for $\kappa$ by
upper bounding $\sqrt{\kappa}$ in \eq{38} by $1$.
\end{pf*}

\begin{pf*}{Proof of Corollary~\ref{cor1}}
We apply Theorem~\ref{thm1} to the Stein coupling
\[
\bigl(\Sigma^{-1/2}W, \Sigma^{-1/2}W',
\Sigma^{-1/2}G\bigr).
\]
The first two terms in the bound \eq{7} are obtained by $\llvert \Sigma
^{-1/2}G\rrvert \leq s_2 \llvert G\rrvert $ and
$\llvert \Sigma^{-1/2}D\rrvert \leq s_2 \llvert D\rrvert $. For the last three terms, we first
observe that for a fixed $d\times d$ orthogonal matrix $U=(U_{1\cdot
},\dots, U_{d\cdot})^t$, a $d$-dimensional random vector $V$ and a
random variable $X$,
\begin{eqnarray*}
\sum_{i=1}^d \Var\E^W\bigl
\{(UV)_i X\bigr\}&=&\sum_{i=1}^d
\Var\E^W\bigl\{ U_{i\cdot}^t V X\bigr\}=\sum
_{i=1}^d \Var\bigl\{U_{i\cdot}^t
\E^W\{VX\}\bigr\}
\\
&=&\sum_{i=1}^d U_{i\cdot}^t
\Cov\bigl(\E^W \{VX\}\bigr)U_{i\cdot}=\Tr \bigl( U \Cov\bigl(
\E^W \{VX\}\bigr) U^t \bigr)
\\
&=&\Tr \bigl( \Cov\bigl(\E^W\{VX\}\bigr) \bigr)=\sum
_{i=1}^d \Var\bigl\{\E ^W
\{V_i X\}\bigr\}.
\end{eqnarray*}
Therefore, $B_1, B_2$ and $B_3$ remain unchanged if we replace $G$ and
$D$ by $UG$ and $UD$. Next, we write $\Sigma^{-1/2}=U\Lambda U^t$ where
$U$ is an orthogonal matrix and $\Lambda$ is a diagonal matrix whose
components are bounded by $s_2$ by definition. Finally, the last three
terms in the bound \eq{7} are obtained by
\begin{eqnarray*}
\sum_{i=1}^d \Var\E^W \bigl
\{\bigl(\Sigma^{-1/2}V\bigr)_i X\bigr\}
 &=&\sum
_{i=1}^d \Var\E^W \bigl\{\bigl(U\Lambda
U^t V\bigr)_i X\bigr\}
\\
&=&\sum_{i=1}^d \Var\E^W
\bigl\{\bigl(\Lambda U^t V\bigr)_i X\bigr\} 
\\
&\leq&
s_2^2 \sum_{i=1}^d
\Var\E^W \bigl\{\bigl(U^t V\bigr)_i X\bigr\}
\\
&=&s_2^2 \sum_{i=1}^d
\Var\E^W \{V_i X\}.
\end{eqnarray*}
\upqed\end{pf*}

\begin{pf*}{Sketch of the proof for Remark~\ref{rem2}}
Let $U$ and $V$ be
uniform on $[0,1]$, independent of each other and all else. Under the
conditions of Remark~\ref{rem2}, we have from Taylor expansion that
\begin{eqnarray*}
0 & =& \lambda^{-1}\IE\bigl\{f\bigl(W'\bigr)-f(W)\bigr\}
\\
& =& \lambda^{-1}\IE\sum_{i=1}^d
\bigl(W_i'-W_i\bigr) f_i(W)
+ \lambda^{-1}\IE\sum_{i,j=1}^d U
D_iD_j f_{ij} (W+UVD)
\\
& =& -\IE\sum_{i=1}^d W_i
f_i(W) + \IE\sum_{i,j=1}^d
G_iD_j f_{ij} (W)
\\
&&{} - 2 \IE\sum_{i,j=1}^d U
G_iD_j \bigl(f_{ij}(W)-f_{ij}
(W+UVD)\bigr).
\end{eqnarray*}
Therefore,
\begin{eqnarray*}
\E\bigl\{\Delta f(W)- W^t\nabla f(W)\bigr\}
&=& \sum_{i,j=1}^d \E\bigl\{(
\delta_{ij}- G_{i}D_{j})f_{ij}(W)\bigr
\}
\\
&&{}+2 \IE\sum_{i,j=1}^d U
G_iD_j \bigl(f_{ij}(W)-f_{ij}
(W+UVD)\bigr)\\
&=:& R'_1 - R'_2.
\end{eqnarray*}
The quantity $R'_1$ is the same as $\int_0^1 R_1(s)\,\mathrm{d}s$ in the proof of
Theorem~\ref{thm1}.
The quantity $R'_2$ contains an additional integration step as compared to
$\int_0^1 R_2(s)\,\mathrm{d}s$ of Theorem~\ref{thm1}, but can be bounded in very
much the same way (up to
different constants).
\end{pf*}

\section{Some Stein couplings} \label{sec5}

In this section, we describe some known coupling constructions as
multivariate Stein couplings for reference.

\subsection{Multivariate exchangeable pairs} \label{sec6}

Chatterjee and Meckes \cite{Chatterjee2008a} and Reinert and R{\"o}llin \cite{Reinert2009} introduced the exchangeable
pairs method for random vectors, which are instances of Stein
couplings. Assume
that
$(W, W')$ is an exchangeable pair of
$d$-dimensional random vectors such that

\begin{equation}
\label{39} \E ^W \bigl(W'-W\bigr) = -\Lambda W
\end{equation}
for some invertible $(d\times d)$-matrix $\Lambda$. It is straightforward to
check that
\[
\bigl(W,W',G\bigr) := \bigl(W,W',\tfrac{1}{2}
\Lambda^{-1}\bigl(W'-W\bigr) \bigr)
\]
is a Stein coupling.

Assume $\Var(W) = \Sigma$ is positive definite. Let $\Sigma^{1/2}$ be the
unique positive-definite root of $\Sigma$, and let $\Sigma^{-1/2}$ be its
corresponding unique inverse. It was shown by Reinert and R{\"o}llin \cite{Reinert2009} that
exchangeability of $(W,W')$ implies symmetry of $\^\Lambda =
\Sigma^{-1/2}\Lambda\Sigma^{1/2}$. Let therefore $O$ be an orthonormal matrix
and let $L$ be a positive diagonal matrix such that $\^\Lambda = OLO^t$.
Define $\^W = O^t\Sigma^{-1/2}W$, $\^W' = O^t\Sigma^{-1/2}W'$. It follows from
\eq{39} that
\[
\E^{\^W} \bigl(\^W'-\^W\bigr) =
-L\^W.
\]
We could therefore -- in principle -- restrict ourselves to $(W,W')$ that are
uncorrelated with \eq{39} being true for diagonal $\Lambda$. However, it is
 often much easier to work with the unstandardized $W$, as
$\Sigma^{-1/2}$ and $O$ are typically  difficult to calculate in practice.

\subsection{Multivariate size bias couplings} \label{sec7}

This coupling was considered by Goldstein and Rinott \cite{Goldstein1996}. Let $Y$ be a non-negative
$d$-dimensional random vector with mean $\mu$ and covariance matrix $\Sigma$.
For each $i=\{1,\ldots,d\}$, let $Y^i$ be defined on the same probability space
as $Y$ and have $Y$-size biased distribution in direction $i$, that is,
\[
\E\bigl\{Y_i f(Y)\bigr\}=\mu_i\E f
\bigl(Y^i\bigr)
\]
for all functions $f$ such that the expectations exist. Let $K$ be uniformly
distributed over $\{1,2,\ldots,d\}$, independent of all else, and let $e_K$ be the $d$-dimensional unit
vector in direction~$K$. Then
\[
\bigl(W,W',G\bigr) :=\bigl(Y-\mu,Y^{K}-\mu,d
\mu_{K}e_K\bigr)
\]
is a Stein coupling.

\subsection{Local dependence}\label{sec8}

A refined version of this dependence was considered by Rinott and Rotar{$'$} \cite{Rinott1996}. Let
$(X_i)_{i\in\mathcal{I}}$ be a collection of centered $d$-dimensional random
vectors for some finite index set $\mathcal{I}$. For each $i\in\mathcal{I}$,
assume there is a set $A_i\subset\mathcal{I}$ such that $X_i$ is independent of
$(X_j)_{j\in A^c_i}$. Let $I$ be uniformly distributed on~$\mathcal{I}$, independent of all else. Then
\begin{equation}
\label{112} \bigl(W,W',G\bigr) := \biggl(\sum
_{i\in\mathcal{I}} X_i, \sum_{i\in\mathcal{I}\setminus
A_I}
X_i, -nX_I\biggr)
\end{equation}
is a Stein coupling.

We have the following corollary of Theorem~\ref{thm1} for the Stein coupling~\eq{112}. The proof is straightforward and therefore omitted here.

\begin{corollary}\label{cor51}
Let
$(X_i)_{i\in\mathcal{I}}$ be a collection of centered $d$-dimensional random
vectors for some finite index set $\mathcal{I}$ with cardinality $n$. For each $i\in\mathcal{I}$,
assume there are sets $A_i\subset B_i\subset\mathcal{I}$ such that $X_i$ is independent of
$(X_j)_{j\in A^c_i}$ and $(X_j)_{j\in A_i}$ is independent of $(X_j)_{i\in B^c_i}$.
Assume further that
\[
\llvert X_i\rrvert \leq \beta,\qquad \bigl\llvert \bigl\{j\in
\mathcal{I}\dvtx  (A_j\cap B_i)\cup(A_i\cap
B_j) \ne \emptyset \bigr\}\bigr\rrvert \leq c_d,
\]
where $c_d$ is a constant only depending on $d$ and $\llvert \cdot\rrvert $ denotes cardinality. Then we have
\[
d_c\bigl(\law (W), \law (Z)\bigr)\leq C_d
\beta^3\, n
\]
for some constant $C_d$ only depending on $d$.
\end{corollary}

Under the condition of the above corollary, the result in Rinott and Rotar{$'$} \cite{Rinott1996} yields the bound
\[
d_c\bigl(\law (W), \law (Z)\bigr)\leq C_d
\beta^3\, n \log n,
\]
which has an additional logarithmic factor.

Note that if the summands are locally dependent, but highly uncorrelated, that is, if $\IE(X_iX_j^t)=0$ ``for many'' $j\in A_i$, it seems difficult to obtain informative bounds from
Rinott and Rotar{$'$} \cite{Rinott1996}. For example, if one tries to apply their Theorem~2.1 to dense random graphs in Section~\ref{sec31}, in order for $\chi_1$ in their~(2.2) to be small, $U_j$ in their Theorem~2.1 has to be of order $n^2$ (recall $n$ is the number of vertices in Theorem~\ref{thm2}), this makes $A_1$ in their Theorem~2.1 too large for their bound~(2.3) to converge to~$0$. In contrast, our Theorem~\ref{thm1} can yield informative bounds in such cases.


\section*{Acknowledgments}
We are grateful to Louis Chen for fruitful discussions that lead to this
research and to Persi Diaconis for the suggestion to study the
joint distribution of inversions and descents in random permutations. We also thank the two anonymous referees for their detailed comments which led to many improvements, in particular the correction of an error in Lemma~\ref{lem3} in an earlier version of the manuscript.

X. Fang was supported by NUS Research Grants C-389-000-010-101 and C-389-000-012-101.
A. R\"ollin was partially supported by NUS Research Grants R-155-000-098-133 and R-155-000-124-112.


\printhistory
\end{document}